\documentclass[12pt,reqno]{amsart}
\usepackage{amssymb,mathrsfs,longtable}

\newcommand{\N}{\mathbb{N}}
\newcommand{\R}{\mathbb{R}}

\newcommand{\B}{\mathcal{B}}
\newcommand{\D}{\mathcal{D}}
\newcommand{\E}{\mathcal{E}}

\newcommand{\G}{\mathcal{G}}
\newcommand{\HH}{\mathcal{H}}
\newcommand{\U}{\mathcal{U}}
\newcommand{\LL}{\mathcal{L}}
\newcommand{\M}{\mathcal{M}}
\newcommand{\NN}{\mathcal{N}}
\newcommand{\cO}{\mathcal{O}}
\newcommand{\cS}{\mathcal{S}}
\newcommand{\T}{\mathcal{T}}

\newcommand{\al}{\alpha}
\newcommand{\be}{\beta}
\newcommand{\de}{\delta}
\newcommand{\ga}{\gamma}
\newcommand{\e}{\varepsilon}
\newcommand{\fy}{\varphi}
\newcommand{\om}{\omega}
\newcommand{\la}{\lambda}
\newcommand{\te}{\theta}
\newcommand{\s}{\sigma}
\newcommand{\ta}{\tau}
\newcommand{\ka}{\kappa}
\newcommand{\x}{\xi}
\newcommand{\y}{\eta}
\newcommand{\z}{\zeta}
\newcommand{\rh}{\rho}
\newcommand{\vr}{\varrho}

\newcommand{\De}{\Delta}

\newcommand{\La}{\Lambda}

\newcommand{\p}{\partial}
\newcommand{\na}{\nabla}

\newcommand{\lec}{\lesssim}
\newcommand{\gec}{\gtrsim}

\newcommand{\etc}{,\ldots,}
\newcommand{\I}{\infty}

\newcommand{\ti}{\widetilde}

\newcommand{\LR}[1]{{\langle #1 \rangle}}

\newcommand{\EQ}[1]{\begin{equation}\begin{split} #1 \end{split}\end{equation}}
\newcommand{\Del}[1]{}
\newcommand{\CAS}[1]{\begin{cases} #1 \end{cases}}
\newcommand{\mat}[1]{\begin{pmatrix} #1 \end{pmatrix}}
\newcommand{\pt}{&}
\newcommand{\pr}{\\ &}
\newcommand{\pq}{\quad}
\newcommand{\pn}{}
\newcommand{\prq}{\\ &\quad}

\newcommand{\prQQ}{\\ &\qquad\qquad}

\numberwithin{equation}{section}

\newtheorem{thm}{Theorem}[section]
\newtheorem{cor}[thm]{Corollary}
\newtheorem{lem}[thm]{Lemma}
\newtheorem{prop}[thm]{Proposition}
\theoremstyle{remark}

\newtheorem{defn}[thm]{Definition}

\newcommand{\mb}{\mathfrak{m}_\phi}
\newcommand{\tm}{\tilde{\mathfrak{m}}_\de}
\newcommand{\diff}[1]{{\triangleleft #1}}
\newcommand{\pa}{\triangleright}
\newcommand{\ok}{\overline{k}}
\newcommand{\uk}{\underline{k}}
\newcommand{\Dm}{\mathscr{D}}
\newcommand{\Sm}{\mathscr{S}}
\newcommand{\Rm}{\mathcal{R}}
\newcommand{\gr}[1]{\lceil #1 \rfloor}
\newcommand{\Str}{{\operatorname{Str}}}
\newcommand{\bde}[1]{(#1 \wedge\de)}

\newcommand{\f}{\frac}

\begin{document}

\author{K.~Nakanishi}
\address{Department of Mathematics, Kyoto University\\ Kyoto 606-8502, Japan}
\email{n-kenji@math.kyoto-u.ac.jp}

\author{W.~Schlag}
\address{Department of  Mathematics, The University of Chicago\\ Chicago, IL 60615, U.S.A.} 
\email{schlag@math.uchicago.edu} 

\thanks{The second author was supported in part by the National Science Foundation,  DMS-0617854  as well as by a Guggenheim fellowship.}

\title[Invariant manifolds for solitons of NLKG]{Invariant manifolds around soliton manifolds \\ for the nonlinear Klein-Gordon equation}

\begin{abstract}
We construct center-stable and center-unstable manifolds, as well as stable and unstable manifolds, for the nonlinear Klein-Gordon equation with a focusing energy sub-critical nonlinearity, 
associated with a family of solitary waves which is generated from any radial stationary solution by 
the action of all Lorentz transforms and spatial translations. The construction is based on the graph transform (or Hadamard) approach, which
requires less spectral information on the linearized operator, and less decay of the nonlinearity, than the Lyapunov-Perron method employed previously in this context.  
The only assumption on the stationary solution is that the kernel of the linearized operator is spanned by its spatial derivatives, which is known to hold for the ground states.  
The main novelty of this paper lies with the fact that the graph transform method is carried out in the presence of modulation parameters corresponding to the symmetries. 
\end{abstract}

\maketitle
\tableofcontents

\section{Introduction}
Consider the focusing nonlinear Klein-Gordon equation (NLKG) on $\R^d$ 
\EQ{ \label{NLKG}
 \ddot u - \De u + u = f(u), \pq u(t,x):\R^{1+d}\to\R}
where $f:\R\to\R$ is a given nonlinearity. 
A typical example is the focusing power nonlinearity 
\EQ{ \label{power f}
 f(u)=|u|^{p-1}u, \pq 3\le p+1<\CAS{\frac{2d}{d-2} &(d\ge 3),\\ \I &(d\le 2).}}
The lower bound can in principle be reduced to $p>1$, but we assume $p\ge 2$ to avoid technical and non-essential complications in the nonlinear estimates. 

The equation preserves the total energy and momentum
\EQ{ \label{def EP}
 E(u) := \int_{\R^d} \big[\frac{|\dot u|^2+|\na u|^2+|u|^2}{2}-f^{(-1)}(u)\big]\,dx, 
 \pq P(u) := \int_{\R^d}\dot u\na u\, dx,}
where $f^{(-1)}:\R\to\R$ is the primitive $f^{(-1)}(a)=\int_0^af(b)\, db$. These quantities are well-defined in the energy space 
\EQ{ \label{def HH}
 \vec u(t):=(u(t),\dot u(t)) \in \HH:=H^1(\R^d)\times L^2(\R^d).}
Throughout the paper, we do not distinguish vertical and horizontal vectors in $\HH$, unless it may lead to any confusion. 

We will consider  NLKG   in the energy space $\HH$, regarding it as a Hamiltonian system. Our goal is to construct a local center-stable manifold of the family of traveling waves generated by the Lorentz transforms and spatial translations acting on a stationary solution. For brevity, we call the latter manifold of traveling waves the {\em soliton manifold}. 
There are two major approaches used in the  construction of center-stable manifolds: the {\em Hadamard method} and the {\em Lyapunov-Perron method}. The former uses the evolution backward and locally in time to find a  flow-invariant graph of the unstable modes in terms of the other components (the Hadamard approach   
also goes by the name of {\em graph transform} or {\em invariant cones} method). The latter uses the evolution forward and globally in time to find an initial adjustment of the unstable modes so that they remain small forever. 

Bates and Jones~\cite{BJ} developed the Hadamard method  in the general setting of an ODE of the form $\dot x= Ax + f(x)$ where $A$ is an (unbounded)
operator on some Banach space $X$ which generates a continuous semi-group, the nonlinearity $f$ is locally Lipschitz on $X$, satisfies $f(0)=0$, and admits arbitrarily small Lipschitz constants near
the equilibrium~$x=0$.   The spectrum of $A$ is divided in the {\em stable part} with eigenvalues in the left-half plane, the {\em unstable part} which lies
in the right-half plane, and the {\em center part} which lies on the imaginary axis. Assumptions are made on the dimensions of the corresponding spectral
subspaces of $X$, and the   associated flows (if the spaces are infinite-dimensional) so as to represent two main scenarios: the {\em dissipative case} (D) on
the one hand, and the {\em conservative case} (C) on the other hand. For (D) one demands that only the {\em stable} subspace be infinite-dimensional
and that the associated semigroup is exponentially stable. For (C) only the {\em center} subspace is infinite dimensional, which is precisely what occurs in Hamiltonian
problems. 

Bates and Jones then verified that the abstract center-stable manifold which they constructed  in~\cite{BJ} applies to stationary 
solutions of NLKG  under {\it the radial symmetry restriction}, for the power nonlinearity \eqref{power f} with $p<\f{d}{d-2}$, $d\ge 3$, where the upper bound on $p$ was required to ensure that the nonlinearity is locally Lipschitz $H^1\to L^2$. 
They also showed that if the linearized operator has no nonzero radial functions in its kernel, then 
\begin{enumerate}
\item Every solution starting on the center-stable manifold stays there forever in positive times, remaining in a small neighborhood of the stationary solution. 
\item Every solution starting in that small neighborhood, but off the manifold, must exit the neighborhood in  finite positive time. 
\end{enumerate}
The kernel condition holds for the ground state (the positive stationary solution), by work of Weinstein~\cite{Wein}. 
 
Gesztesy, Jones, Latushkin, and Stanislavova \cite{GJLS} demonstrated that the Bates--Jones theory applies to the nonlinear Schr\"odinger equation (NLS) with a spatially localized nonlinearity. Notice that the radial restriction for NLKG prohibits both the spatial translations and the Lorentz transforms, and so the soliton manifold is reduced to a fixed stationary solution. Similarly, the localized nonlinearity of~\cite{GJLS} destroys the scaling, translation, and Galilean invariance, so that the soliton can change only with respect to the phase parameter. Indeed, as we will explain below, moving solitons represent a serious obstacle to the Bates--Jones approach. 

On the other hand, the second author \cite{S} developed the Lyapunov-Perron (LP) method for the ground state of the cubic NLS in~$\R^{3}$, without imposing any symmetry assumptions, but in a weighted $H^{s}$-space (or an unweighted $L^{1}$-based space). In this approach, the soliton is allowed to move. Recently, Beceanu \cite{Bec2} extended the latter work to the critical space $\dot H^{1/2}$ which is bigger than the energy space. 
Finally, and partly based on a novel approach to linear dispersive estimates developed by Beceanu~\cite{Bec1}, 
the authors proved in~\cite{NLKGnonrad} that the LP approach can be carried out for NLKG in the energy space without any symmetry restrictions. 

However, an essential difficulty in applying the LP method to a nonlinear dispersive equation (without dissipation) is that it requires global dispersive estimates, which in turn necessitates   fine spectral information, such as the absence of {\em  threshold resonances} and of so-called {\em spurious\footnote{This refers to eigenvalues which do not result from symmetries of the equation.} eigenvalues}; alternatively, in the presence of such spurious eigenvalues one might hope to invoke the {\em  Fermi golden rule}. Those conditions are in general very hard to check, even for the ground state (apart from the one-dimensional case \cite{KS} where purely analytical arguments are available), although there has been some recent progress in this direction \cite{DS,FMR,MS,CHS}.  

While the LP method requires stronger ingredients, it also leads to more detailed conclusions. More specifically, one obtains 
that solutions starting on the center-stable manifold scatter to the soliton manifold in forward time. In other words, the distinction between the LP approach
and the Hadamard approach is roughly tantamount to the distinction between {\em asymptotic and orbital stability} of solitary waves. 

In this paper, we employ the Hadamard method in the nonradial setting. Our main challenge is to extend the result by Bates and Jones to the family of traveling waves, rather than stationary ones. We therefore have to investigate the dynamics along the soliton manifold as well, which is usually called the modulational analysis in the stability problem of solitons. 
In our setting, the soliton manifold has $2d$ dimensions corresponding to the relativistic momentum and position vectors in $\R^d$. 

Those parameters can be fixed by means of a Lorentz transform which reduces the total momentum to zero, and by using a coordinate moving with the soliton. In doing so,  we encounter a derivative loss due to the translation, i.e.,  a transport term, in the modulated equation for the difference of two solutions, which disables the contraction argument for the graphs in the energy norm. This difficulty is not an artifact of the coordinate choice, but a natural consequence of the two facts that the solitons are translated by the flow, while the translation is not Lipschitz continuous in any Sobolev space. The same problem occurs for any other continuous group action involving a coordinate change, such as scaling or rotation. 

We overcome this difficulty by introducing a {\em nonlinear quasi-distance} in the energy space, for which the spatial translation becomes Lipschitz continuous, while the topology remains the same. Using the contraction mapping principle with this distance for the continuous spectral part, we are able to carry out the Hadamard method in the presence of  modulational parameters.

A more technical issue concerns allowing nonlinearities all the way up to the $H^1$ critical power, i.e., for $p<(d+2)/(d-2)$, while Bates and Jones assumed $p<d/(d-2)$. 
This is easily resolved by using the Strichartz estimate for the free Klein-Gordon equation, and by relaxing some flow-invariance conditions by constant multiples. 

Since the description of dynamics around the manifolds ((1)-(2) above) is also extended to the current setting, we can easily observe that the maximal backward evolution of the center-stable manifold is identical, in a small neighborhood, to the forward trapping set $\T_+$: the collection of initial data for which the solution (of the original NLKG) stays in the small neighborhood for large times. 

In the special case where the soliton manifold is generated from the ground state, we can combine the above result with the one-pass theorem in \cite{NLKGnonrad} as well as the openness and connectedness of the forward scattering set $\cS_+$ and the forward blow-up set $\B_+$, thereby concluding that $\T_+$ separates locally and globally all the solutions with energy at most slightly above the ground state energy into $\cS_+$ and $\B_+$. Therefore, the conclusion of \cite{NLKGnonrad} is extended to the range 
\EQ{
 d\in\N, \quad 1+\frac{4}{d} < p < 1+\frac{4}{d-2},\pq p\ge 2,}
{\em except for the following scattering statement on $\T_+$}:  all solutions in $\T_+$ scatter to the soliton manifold as $t\to\I$. 
This statement was proved in \cite{NLKGnonrad} for $d=p=3$  by means of the LP method using the following {\em gap property} of the linearized operator $L_{+}$: 
\[
(0,1) \cap \s(L_{+})=\emptyset, \text{\ \ and there is no threshold eigenvalue or resonance} 
\]
This is proved (at least for the radial case) in~\cite{CHS}. 
Note that the numerical analysis of~\cite{DS} suggests that the absence of threshold resonances and spurious  eigenvalues {\em fails} for some powers in $(1+4/3,3)$ for $d=3$, where the LP method without any hypothesis (typically the Fermi golden rule) is not so far available. 
Also note that the lower bound $1+4/d$ is required by the proof of the one-pass theorem, but not by the Hadamard construction in this paper, while the Lyapunov-Perron method also needs it in order to work in the energy space. 

To state the main result, we clarify the assumptions on the nonlinearity $f$ and on the stationary solution: 
\EQ{ \label{asm f}
 \pt f\in C^2(\R;\R),\pq 0=f(0)=f'(0), 
 \prq \forall a\in\R,\pq |f''(a)|\lec \CAS{1+|a|^{p-2} &(d\ge 3, 2\le \exists p<\frac{d+2}{d-2})\\ 1+|a|^{p-2} &(d=2, 2\le \exists p<\I)\\ \text{arbitrary} &(d=1).}}
To have $C^1$ manifolds, we assume a bit more regularity: for some $p>2$ in the above range,  
\EQ{ \label{asm f'}
 |f''(a_1)-f''(a_2)| \lec (|a_1-a_2|^{p-2}+|a_1-a_2|)[1+|a_1|^{p-3}+|a_2|^{p-3}].}
These assumptions are satisfied for example by 
\EQ{
 f(u) = \sum_{k:\text{finite}} \la_k |u|^{p_k-1}u, \pq \la_k\ge 0,}
provided that all $p_k>2$ are in the range \eqref{asm f}.  

Let $Q\in H^1(\R^d)$ be a stationary solution of NLKG, i.e., a weak solution of the elliptic PDE 
\EQ{ \label{static eq}
 -\De Q + Q = f(Q).}
Standard arguments imply that $Q\in H^2$ with exponential decay as $|x|\to\I$. For existence, see the classical work by Berestycki, Lions~\cite{BerLions}. 
The action of the  Lorentz transforms and the spatial translations generate the traveling wave family parametrized by the relativistic momentum $\vec p\in\R^d$ and position $\vec q\in\R^d$: 
\EQ{ \label{def Qpq}
 Q(\vec p,\vec q):=Q(x-\vec q+\vec p(\LR{\vec p}-1)|\vec p|^{-2}\vec p\cdot(x-\vec q)),}
so that each traveling wave can be written in the form 
\EQ{
 u(t)=Q(\vec p,\vec q(t)), \pq \f{d}{dt}\vec q(t)=\frac{\vec p(t)}{\LR{\vec p(t)}}.}
For brevity, the spatial translate is denoted also as
\EQ{ \label{def Qc}
 Q_c(x) := Q(x-c).}
The vector form is denoted by 
\EQ{ \label{def vQ}
 \vec Q:=(Q,0), \pq \vec Q(p,q):=(Q(p,q),-\frac{\vec p}{\LR{\vec p}}\cdot\na Q(p,q)),}
and the soliton manifold of $Q$ is defined as
\EQ{ \label{def sol mfd}
 \Sm(Q) := \{\vec Q(\vec p,\vec q)\}_{\vec p,\vec q\in\R^d} \subset \HH,}
which is a $C^1$ manifold of dimension $2d$. The linearized operator at $Q$
\EQ{ \label{def L+}
 L_+:=D^2-f'(Q)=-\De+1-f'(Q), \pq D:=\sqrt{1-\De}}
is self-adjoint on $L^2$ with a finite number of eigenvalues and continuous spectrum $\s_c(L_+)=\s_{ac}(L_+)=[1,\I)$. The translation invariance of NLKG implies that $L_+\na Q=0$. The only assumption on $Q$ in this paper is 
\EQ{ \label{L+ ker}
 L_{+}^{-1}(0) = \mathrm{span}\{\na Q\}.}
This is a well-known property of the ground states. To be more precise, by an argument of Weinstein~\cite{Wein}, it holds for the ground state~$Q$
provided no radial function lies in the kernel of~$L_{+}$.  The latter holds for any subcritical monomial nonlinearity (as well as others), see Lemma~2.3
in~\cite{NLKGrad}, for example. 
Moreover, \eqref{L+ ker} seems to be 
 a natural assumption for any other radial static solution. For non-radial static solutions, we have to include angular derivatives as well, but we do not consider such solutions in this paper. Although we will not explicitly use the radial symmetry of $Q$, the reader may assume it without losing anything throughout the paper. 

\begin{thm}
Let $d\in\N$ and assume that $f$ satisfies \eqref{asm f}. Let $Q$ be a static solution \eqref{static eq} and assume that its linearized operator $L_+$ satisfies \eqref{L+ ker}. Then there is a Lipschitz manifold $\M_{cs}$ in $\HH$ containing the soliton manifold $\Sm(Q)$, with the following properties:
\begin{enumerate}
\item The codimension of $\M_{cs}$ in $\HH$ equals  the total dimension of the eigenspaces of $L_+$ corresponding to negative eigenvalues, which we denote by~$K$. 
\item $\M_{cs}$ is invariant under the forward evolution of NLKG \eqref{NLKG}.
\item $\M_{cs}$ is invariant under spatial translations. 
\item For every Lorentz transform and every $(\vec p,\vec q)$, there is a small neighborhood of $\vec Q(\vec p,\vec q)$ such that the Lorentz transform of any forward global solution starting from $\M_{cs}$ within this neighborhood  remains on  $\M_{cs}$ for all $t>0$. 
\item $\M_{cs}$ is normal at $\vec Q$ to the vector $(-k\rh,\rh)$ for any $\rh\in H^1$ solving $L_+ \rh = -k^2 \rh$ for some $k>0$. In other words, 
\EQ{
 \LR{u_2|\rh} = -k\LR{u_1-Q|\rh} + o(\|u_1-Q\|_{H^1}+\|u_2\|_{L^2}).}
\item For any neighborhood $\cO$ of $\Sm(Q)$, there is a smaller neighborhood $\cO'$, such that every solution starting from $\cO'\cap\M_{cs}$ remains in $\cO\cap\M_{cs}$ for all $t>0$. 
\item There is a neighborhood $\cO$ of $\Sm(Q)$ such that every solution starting from $\cO\setminus\M_{cs}$ exits $\cO$ in finite positive time.
\item If in addition $f$ satisfies \eqref{asm f'}, then $\M_{cs}$ is $C^{1,\al}$, where $\al=\min(1,p-2)$. 
\end{enumerate}
\end{thm}

The corresponding statement for a center-unstable manifold follows simply by the time inversion, so we omit it. However, in the proof we will actually consider the center-unstable manifold, for which the forward evolution is used as a contraction mapping in the Hadamard method.  The {\em center manifold} is obtained by intersecting the center-stable with the center-unstable manifold. 
It is of codimension~$2K$ and is bi-invariant.

Properties (6), (7) characterize $\M_{cs}$ as the set of solutions which stay close to $\Sm(Q)$ for all $t\ge0$. Since the Lyapunov-Perron method looks for such solutions from the beginning, it will yield a subset of $\M_{cs}$, and indeed the same manifold (locally), provided that the codimension is the same. An advantage of the Lyapunov-Perron method is that it implies the scattering to the soliton manifold for the solutions on $\M_{cs}$ (cf.~\cite{NLKGnonrad}). It will be interesting to see what happens when some spectral condition breaks down, e.g., if there is a threshold resonance.

We also obtain a stable and unstable manifold theorem. Recall the definition of $K$ from the previous theorem. 

\begin{thm}
\label{thm:stable}
Under the assumptions of the previous theorem, there exist Lipschitz manifolds $\M_u$ and $\M_s$ of dimensions $K+d$ with the following properties:
\begin{enumerate}
\item $\M_{s}$ is invariant under the forward evolution of the NLKG \eqref{NLKG}.
\item $\M_{s}$ is invariant under spatial translations. 
\item every solution starting on $\M_s$ converges exponentially to $\vec Q(\cdot-c(t))$ as $t\to\I$ where $\dot c(t)\to0$ exponentially as $t\to\I$. 
\item there exists $\de>0$ small such that $\M_s$ is a Lipschitz graph over $B_\de(0)\times \R^d$, with $B_\de(0)$ being a $\de$-ball in the eigenspace
corresponding to the negative eigenvalues, and the $\R^d$-component deriving from spatial translations.  
\item $\M_u$ is obtained from $\M_s$ by reversing time. 
\end{enumerate}
\end{thm}

\section{Preliminaries}
Here we fix some notation. For any two elements $v^0,v^1$ in a normed space $V$, the ordered pair and their difference are denoted by 
\EQ{ \label{def diff}
 v^\pa:=(v^0,v^1)\in V^2, \pq \|v^\pa\|_V=\|v^0\|_V+\|v^1\|_V, \pq \diff v^\pa:=v^1-v^0,}
respectively. More generally, for any map $M:V\times W\times\cdots$ and elements $v^j\in V$, $w^j\in W$,\dots, the mapped pair is denoted by 
\EQ{
 M(v^\pa,w^\pa,\dots):=(M(v^0,w^0,\dots),M(v^1,w^1,\dots)).}
For any $R\in\R$ and $\de>0$, the minimum is denoted by 
\EQ{ \label{def min}
 \bde{R}:=\min(R,\de).}
 As usual, $a\lec b$, $a\gtrsim b$,  and $a\simeq b$ involve implicit multiplicative constants. 
 
\subsection{Equation and spectrum}
The energy space $\HH=H^1\times L^2\subset(L^2)^2$ is endowed with the usual inner product 
\EQ{ \label{def prdH}
 \LR{\fy,\psi}_\HH := \int_{\R^d} \big[ \na\fy_1\cdot\na\psi_1+\fy_1\psi_1+\fy_2\psi_2 \big]\,   dx}
and the $L^2$ duality coupling 
\EQ{ \label{def prdL2}
 \LR{\fy|\psi} :=\int_{\R^d} \big[ \fy_1(x)\psi_1(x)+\fy_2(x)\psi_2(x) \big]\, dx,}
as well as the symplectic form 
\EQ{ \label{def om}
 \om(\fy,\psi):=\LR{J\fy|\psi}=\int_{\R^d} \big[ \fy_2(x)\psi_1(x)-\fy_1(x)\psi_2(x) \big] \, dx,}
where $J$ is the skew-symmetric matrix 
\EQ{ \label{def J}
 J := \mat{0 & 1 \\ -1 & 0}, \pq J^2=-\mat{1 & 0 \\ 0 & 1}.}
Let $\LL$ be the self-adjoint operator on $L^2$ with domain $H^2\times L^2$
\EQ{ \label{def LL}
 \LL := \mat{L_+ & 0 \\ 0 & 1} = \mat{-\De+1-f'(Q) & 0 \\ 0 & 1}.}
Its free version is denoted by 
\EQ{ \label{def D}
 \D := \mat{D^2 & 0 \\ 0 & 1} = \mat{-\De+1 & 0 \\ 0 & 1}.}
Then the linearized equation around $\vec Q=(Q,0)$ is 
\EQ{
 v_t = J\LL v.}
The spectrum of $J\LL$ is given in terms of that of $L_+$:
\[
\s(J\LL) = \pm \sqrt{-\s(L_{+})}
\]
Since $f'(Q)$ is bounded and exponentially decreasing, there are $0<\uk\le 1<\ok$ and a finite set $K\subset[\uk,\ok]$ such that 
\EQ{ \label{spec L+}
 \s(L_+)\setminus[\uk^2,\I)=\{0\}\cup\{-k^2\mid k\in K\}.}
With slight abuse of notation, we let $K$ count the multiplicity of each negative eigenvalue $-k^2$ as well. For each $k\in K$, let $\rh_k\in\cS(\R^d)$ be an eigenfunction satisfying 
\EQ{ \label{def rk}
 L_+\rh_k=-k^2\rh_k, \pq \|\rh_k\|_2=1.}
The (generalized) eigenfunctions of $J\LL$ are 
\EQ{ \label{def gk}
 \pt J\LL g_{k\pm}=\pm k g_{k\pm}, \pq J\LL\na\vec Q = 0, \pq J\LL J\na\vec Q = -\na\vec Q, \pq g_{k\pm}:=\mat{1 \\ \pm k}\frac{\rh_k}{\sqrt{2k}},}
which satisfy 
\EQ{ \label{def H}
 \pt \om(\p_\al\vec Q,J\p_\be\vec Q)=H_{\al,\be}(Q):=\LR{\p_\al Q|\p_\be Q}, 
 \pq \om(g_{k+},g_{k-})=1.}
The corresponding symplectic (spectral) decomposition takes the form 
\EQ{ \label{spec decp}
 v \pt= \sum_{\pm, k\in K} \la_{k\pm}g_{k\pm} + \mu\cdot\na\vec Q + \nu\cdot J\na \vec Q + \ga, 
 \pr \la_{k\pm}:=P_{\pm k}v :=\om(v,\pm g_{k\mp}), 
 \pr \mu:=P_\mu v:=H(Q)^{-1}\om(v,J\na\vec Q), \pq \nu:=P_\nu v:=H(Q)^{-1}\om(v,-\na\vec Q),
 \pr \ga:=P_\ga v:=v-\sum_{\pm,k}\la_{\pm k}g_{k\pm} - \mu\cdot\na\vec Q - \nu\cdot J\na\vec Q.}
We will use the following projections as well: 
\EQ{ \label{spec combi}
 \pt v_\pm:=P_\pm v:=\sum_{k\in K}\la_{k\pm}g_{k\pm}, 
 \pq v_d:=P_dv:=v-P_\ga v, 
 \pr v_0:=P_0v:=\na\vec Q\cdot\mu+J\na\vec Q\cdot\nu, 
 \pq v_{\ge 0}:=P_{\ge 0}v:=v-v_-,
 \pr v_{\ga+}:=P_{\ga+}v:=\ga+v_+, 
 \pq v_{0\pm}:=P_{0\pm}v:=v_0+v_\pm,}
and the corresponding subspaces $\HH_\pm:=P_\pm(\HH)$. 
Fixing a small number 
\EQ{ \label{def ka}
 0 < \ka \ll \uk,}
we define the energy norm on $\HH$ to be  
\EQ{ \label{def E}
 \|v\|_E^2 := \sum_{k\in K}|\la_k|^2 + |\nu|^2 + \ka^2|\mu|^2 + \LR{\LL \ga|\ga} \simeq \|v\|_{\HH}^2,}
where the final equivalence follows from the orthogonality of $\ga_1$:
\EQ{
 0=\LR{\ga_1|\na Q}=\LR{\ga_1|\rh_k} \qquad\forall\; k, }
together with \eqref{L+ ker}, since $\LR{\LL\ga|\ga}=\LR{L_+\ga_1|\ga_1}+\|\ga_2\|_2^2$.

Let $\vec u\in C(I;\HH)$ be a solution \eqref{NLKG}, then $v=(v_1,v_2):=\vec u(x+c)-\vec Q$ solves 
\EQ{ \label{eq0 v}
 \pt v_t = J\LL v + \dot c\cdot(\na \vec Q+\na v) + \vec N(v), \pq \vec N(v)=(0,N(v)),}
where $N:\HH\to H^{-1}$ carries the superlinear part 
\EQ{ \label{def N}
 N(v):=f(Q+v_1)-f(Q)-f'(Q)v_1=o(v_1).}
We remark that $v_2\not=\dot v_1$ unless $\dot c=0$.

The conserved momentum can be rewritten as 
\EQ{ \label{def P}
 P(u):=\om(\vec u,\na\vec u)/2=\om(v,\na\vec Q)+\om(v,\na v)/2.}
By means of a Lorentz transform, we can reduce the dynamics near the soliton manifold $\Sm(Q)$ to the invariant subspace 
\EQ{ \label{def H0}
 \HH_0:=\{\vec u\in\HH \mid P(u)=0\}.}
Furthermore, we can restrict $v$ to the subspace 
\EQ{ \label{def Hperp}
 \HH_\perp:=\{v\in\HH \mid \om(v,\na\vec Q)+\om(v,\na v)/2=0,\ \om(v,J\na\vec Q)=0\},}
by choosing $\dot c$ so that 
\EQ{
 0=\p_t\om(v,J\na\vec Q)=(H(Q)-\LR{\na^2Q|v_1})\dot c+\om(v,\na\vec Q).}
Hence the evolution for small $v$ on $\HH_\perp$ is given by 
\EQ{ \label{eq1 v}
 \pt v_t = J\LL v + A(v)\cdot\na( \vec Q+ v) + \vec N(v),
 \prq A(v):=(H(Q)-\LR{\na^2Q|v_1})^{-1}\om(v,\na v)/2.}
This is a first-order autonomous equation in $\HH$ with the superlinear term 
\EQ{ \label{def M}
 M(v):=A(v)\cdot\na( \vec Q+ v) +\vec N(v).}

In order to implement the Hadamard method, we need to localize the nonlinear part $M(v)$ near $0$ so that it becomes a {\em small Lipschitz term globally in the energy space}~$\HH$. It seems extremely hard to do this keeping the above orthogonal structure, since the acceleration or the modulation term is naturally unbounded, unless the linearized operator is modified depending on the distance of $v$ from $0$. 
Therefore we will not enforce the orthogonality conditions, but instead solve a localized version of the above autonomous equation in the whole energy space~$\HH$. After constructing a center-unstable manifold by the Hadamard method for the localized equation, we will restrict that manifold to the subspace $\HH_\perp$ in a small neighborhood of $0$ to obtain a center-unstable manifold for the true equation. 

In the case of the unstable manifold, the exponential decay of $v$ as $t\to-\I$ ensures that the manifold for the localized equation around $0$ falls into $\HH_\perp$, so that we can automatically get the manifold of the true equation.

\subsection{Mobile distance}
The most serious obstacle to carrying out the graph transform method in the nonradial setting results from the contraction step in the construction
of the invariant graphs, where the presence of the unbounded translation term causes problems. 
To remedy this, we introduce the {\it mobile distance} on $\HH$. Heuristically speaking, the standard $L^p$ or Sobolev-type norm is too tight for ``horizontal motion" $\fy\mapsto \fy(\cdot+x_0)$ compared with ``vertical motion" $\fy\mapsto \la\fy$. The mobile distance makes translation just as easy as amplification, without changing the topology.  

\begin{defn}
For any continuous increasing function $\phi:[0,\I)\to[0,\I)$ satisfying 
\begin{enumerate}
\item $\phi(a)\ge a$.
\item $a\le 2b\implies \phi(b)\le 4\phi(a)$, 
\end{enumerate}
the mobile distance $\mb:\HH\times\HH \to[0,\I)$ is defined by 
\EQ{ \label{def mobile}
 \mb(v^0,v^1)^2:=\inf_{q\in\R^d,\ j=0,1} \; \pt \|v^{1-j}-v^{j}(\cdot-q)\|_E^2 +|q|^2\phi(\|v^j\|_E)^2 .}
\end{defn}
Obviously, the infimum in \eqref{def mobile} is attained at some $q\in\R^d$. $\mb$ is not really a distance, but a quasi-distance on $\HH$. More precisely, we have 
\begin{prop} \label{prop mb}
$\mb$ in \eqref{def mobile} is a complete quasi-distance on $\HH$, satisfying
\begin{enumerate}
\item $\mb(v^0,v^1)\ge 0$ where the equality holds iff $v^0=v^1$.
\item $\mb(v^1,v^0)=\mb(v^0,v^1)$.
\item $\mb(v^0,v^1)\le C_d\,[\mb(v^0,v^2)+\mb(v^2,v^1)]$ for some absolute constant $C_d>0$. 
\item If $\mb(v^m,v^n)\to 0$ ($n,m\to\I$) then $v^n$ converges in $\HH$.
\end{enumerate}
Moreover, it satisfies with some absolute constant $C>0$, 
\EQ{ \label{est on mob}
 |\|v^0\|_\HH-\|v^1\|_\HH|+\|D^{-1}(v^0-v^1)\|_{\HH} \le C\mb(v^0,v^1) \le C^2\|v^0-v^1\|_\HH,}
where $D:=\sqrt{1-\De}$. These constants, $C$ and $C_d$, do not depend on the choice of $\phi$. 
\end{prop}
Hence $\mb$ defines the same topology as $\HH$, differing only in terms of uniformity. For example,  for any $\fy\in\HH$ we have 
\EQ{
 \lim_{n\to\I} \mb(\fy e^{inx_1},-\fy e^{inx_1})=O(1),} 
since $-\fy e^{inx_1}=\fy e^{in(x_1+\pi/n)}$, whereas $\|\fy e^{inx_1}-(-\fy e^{inx_1})\|_\HH=O(n)$ unless $\fy=0$. 
\begin{proof}[Proof of Proposition \ref{prop mb}]
(1) and (2) are obvious. For the left-most term of \eqref{est on mob}, and with $\tau_q v:= v(\cdot-q)$, 
\EQ{
 |\|v^0\|_\HH-\|v^1\|_\HH| \pt\le \inf_q \min(\|v^0-\tau_q v^1\|_\HH,\|v^1- \tau_q v^0 \|_\HH) \pr\lec \mb(v^0,v^1),}
for the second term, 
\EQ{
 \|D^{-1}(v^0-v^1)\|_{\HH} \pt\le \|D^{-1}(v^0- \tau_q v^1 )\|_{\HH}+\|D^{-1}(\tau_q v^1 -v^1)\|_{\HH}
 \pr\lec \|v^0-\tau_q v^1 \|_\HH + |q|\|\na D^{-1} v^1\|_{\HH}
 \pr\lec \|v^0- \tau_q v^1 \|_E + |q|\phi(\|v^1\|_E),}
while the right bound in \eqref{est on mob} is obvious by choosing $q=0$. 
Next we prove the quasi-triangle inequality. For any $v^0,v^1,v^2\in\HH$, there are $q^1,q^2\in\R^d$ such that 
\EQ{
 \mb(v^0,v^j) \simeq \|v^0- \tau_{q^j} v^j \|_\HH + |q^j|\phi(\min(\|v^0\|_\HH,\|v^j\|_\HH)), \pq(j=1,2),}
since the $\HH$ and  $E$ norms are equivalent and the $\HH$ norm is translation invariant. 
If $\|v^0\|_\HH\ll \min(\|v^1\|_\HH,\|v^2\|_\HH)$ then $\mb(v^0,v^j)\simeq\|v^j\|_\HH$ and so the quasi-triangle inequality is obvious. Otherwise, 
\EQ{
 \mb(v^1,v^2)\pt\lec \|v^1-v^2(\cdot-q^2+q^1)\|_\HH + |q^1-q^2|\phi(\min(\|v^1\|_\HH,\|v^2\|_\HH))
 \pr\lec \|v^1(\cdot-q^1)-v^2(\cdot-q^2)\|_\HH + \sum_{j=1,2}|q^j|\phi(\min(\|v^0\|_\HH,\|v^j\|_\HH))
 \pr\lec \mb(v^0,v^1) + \mb(v^0,v^2).}
To prove the completeness, let $v^n$ be Cauchy in $\mb$. Then so is $D^{-1}v^n$ in $\HH$ by \eqref{est on mob}. Hence, $v^n$ converges to some $v\in D\HH$. \eqref{est on mob} implies that $\|v^n\|_\HH$ converges. We may assume that this limit is positive, since otherwise the convergence to $0$ is obvious. Since $v^n$ is bounded in $\HH$, it converges weakly to $v$ in $\HH$. Passing to a subsequence, we may find $q_n\in\R^d$ such that 
\EQ{
 \|\tau_{q_n} v^n -v^{n+1}\|_\HH + |q_n| < 2^{-n}.}
Let $c_n=\sum_{k\ge n}q_n$, then $c_n\to 0$ and 
\EQ{
 \|v^n(\cdot-c_n)-v^{n+1}(\cdot-c_{n+1})\|_\HH < 2^{-n},}
which implies $\tau_{c_n}v^n\to v$ strongly in $\HH$, whence also  $v^n\to v$ strongly in $\HH$. 
\end{proof}

We apply the mobile distance only to the continuous spectrum part because, on the one hand, the discrete spectral part is finite dimensional and smooth, and on the other hand, the linearized energy is conserved only on the continuous spectrum. Choose positive constants $\de,C_0,C_1,C_2$ such that 
\EQ{
 0< C_2 \de \ll 1 \ll C_0 \ll C_1 \ll C_2.}
The required smallness of $C_2\de$, $1/C_0$, $C_0/C_1$ and $C_1/C_2$ is implicit in the following arguments, but only in terms of $d$, $f$, $Q$ and $\ka$. Henceforth, we shall regard those $C_j$ as being fixed constants and ignore the dependence on them unless it is important, while we regard $\de$ as a small parameter (with the smallness depending on $C_2$), keeping track of its impact on the estimates. 

The quasi-distance $\tm:\HH\times \HH \to[0,\I)$ is defined by 
\EQ{ \label{def tm}
 \pt \tm(v^0,v^1)^2:=\|P_d(v^0-v^1)\|_E^2 + \mathfrak{m}_{\phi_\de}(P_\ga v^0,P_\ga v^1)^2,}
where $\phi_\de(a):=\phi(a/\de)$ with a fixed $\phi\in C^\I(\R)$ satisfying 
\EQ{ \label{def phi}
 \phi(a)=\CAS{1 &(a\le C_2)\\ a &(a\ge 2C_2)},\pq 0\le \phi'\le 1.}
We will localize the equation for $v$ within distance $O(\de)$ from $0$, such that the evolution outside of it becomes purely linearized. 
We have chosen $\phi_\de$ such that the ``fare" is purely proportional to the translation distance within the nonlinear region, but there is an additional fee for ``excessive weight" over $O(\de)$. 
It is easy to see that $\tm$ has the same properties as $\mb$ in Proposition \ref{prop mb}. 

\section{Construction of manifolds for a localized equation}
In this section we construct a global center-unstable manifold for a equation of $v$ with localized nonlinearity around $0\in\HH$. The manifold obeys the original flow only on the subset $\HH_\perp$ in a small neighborhood. 
\subsection{Localization of the equation}
Let $\chi\in C_0^\I(\R)$ be a non-negative symmetric decreasing function satisfying $\chi(t)=1$ for $|t|\le 1$ and $\chi(t)=0$ for $|t|\ge 2$. Let 
\EQ{ \label{def chide}
 \chi_\de(v) := \chi(\|v\|_\HH^2/\de^2).}

We will construct a center-unstable manifold near $0$ for the equation of $v$ with the nonlinearity localized within $O(\de)$ distance from $0$
\EQ{ \label{eq v}
 \pt v_t = J\LL v + M_\de(v), \pq 
 \pq M_\de(v):=\chi_\de(v)[A(v)\cdot\na(\vec Q+ v) +\vec N(v)].}
Hence each component in the spectral decomposition solves 
\EQ{ \label{eq v in spec}
 \pt \p_t\la_{k\pm} = \pm k \la_{k\pm} + P_{\pm k} M_\de(v),
 \pr \p_t\mu = -\nu + P_\mu M_\de(v), 
 \pq \p_t\nu = P_\nu M_\de(v), 
 \pr \p_t\ga = J\LL \ga + P_\ga M_\de(v).}
\begin{lem} \label{lem:bd v}
The equation \eqref{eq v} is globally wellposed in $\HH$, and for any solution $v$, 
\EQ{ \label{inc est}
 \pt \sup_{|t|\le 1}\|v(t)\|_E \lec \|v(0)\|_E, 
 \pr \sup_{|t|\le 1}\|v_d(t)-e^{J\LL t}v_d(0)\|_E \lec \bde{\|v(0)\|_E}^2, 
 \pr \sup_{|t|\le 1}|\|\ga(t)\|_E^2-\|\ga(0)\|_E^2| \lec \bde{\|v(0)\|_E}^3.} 
\end{lem}
\begin{proof}
Let $v$ be a local solution around $t=0$, and let $w(t,x)=v(t,x-c)$, where $c$ is the solution of 
\EQ{
 \dot c=\chi_\de(v)A(v),\pq c(0)=0.}
Let $\ta_c$ be the translation operator 
\EQ{ \label{def ta}
 \ta_c\fy(x)=\fy(x-c),}
then the equation for $(w,c)$ is given by
\EQ{ \label{eq wc}
 \pt \dot w = J\D w + F(w,c), \pq \dot c = B(w,c),}
with the nonlinear terms $F$ and $B$, defined by
\EQ{ \label{def BF}
  B(w,c) \pt:= \chi_\de(w)A_c(w), 
 \pq F(w,c) := \mat{B(w,c)\cdot\na Q_c \\ f'(Q_c)w_1 + \chi_\de(w)N_c(w)},}
where $A_c,N_c$ are translates of $A,N$: 
\EQ{ \label{def ANc}
 \pt A_c(w):=A(\ta_c^*w)=(H(Q)-\LR{\na Q_c|w_1})^{-1}\om(w,\na w)/2, 
 \pr N_c(w):=\ta_cN(\ta_c^*w)=f(Q_c+w_1)-f(Q_c)-f'(Q_c)w_1.}
Choosing some appropriate Strichartz norm, for example 
\EQ{ \label{def Str}
 \|w\|_\Str:= \CAS{ \|w\|_{L^\I_t \HH_x} + \|w_1\|_{L^{p}_{t} L^{2p}_x} &(d\ge 3),\\ \|w\|_{L^\I_t \HH_x} &(d\le 2),}}
with $p=\frac{d+2}{d-2}$, we have by the Strichartz estimate for the free Klein-Gordon equation 
\EQ{
 \pt \|w_1\|_{\Str(0,T)} \lec \|w(0)\|_\HH + \|F\|_{L^1_t\HH(0,T)},
 \pq \|c\|_{L^\I(0,T)} \lec |c(0)|+T\|B\|_{L^\I(0,T)},}
where the nonlinear terms are estimated by H\"older 
\EQ{ \label{est diff BF}
 \pt F(0,c)=0, \pq B(0,c)=0, 
 \pr |\diff B(w^\pa,c^\pa)| \lec [|\diff c^\pa| \bde{\|w^\pa\|_\HH}^2 + \bde{\|\diff w^\pa\|_\HH}] \bde{\|w^\pa\|_\HH},
 \pr \|\diff F(w^\pa,c^\pa)\|_{L^1_t\HH(0,T)} \lec T[\|\diff c^\pa\|_{L^\I} \|w^\pa\|_{\Str}  + \|\diff w^\pa\|_{\Str}]
 \pr\hspace{120pt}+ [\|\diff c^\pa\|_{L^\I} + \|\diff w^\pa\|_{\Str}]\bde{\|w^\pa\|_{\Str}} ,}
on the time interval $0<t<T\ll 1$. 
Hence if $\de,T>0$ are small enough, we obtain a local solution of $(w,c)$ on $(0,T)$ in $\HH\times\R^d$ by the standard iteration. By Gronwall, it is extended to any finite time intervals. In particular, 
\EQ{ \label{bd wc}
 \pt \|w_{1}\|_{\Str(-1,1)} \lec \|w(0)\|_\HH, \pq |c|_{L^\I(-1,1)} \lec |c(0)| + \bde{\|w(0)\|_\HH}^2,}
and moreover, if $\|w(0)\|_\HH\gg\de$ then $\|w(t)\|_\HH\gg\de$ for $|t|\le 1$ and so, $c(t)=0$ and $F=(0,f'(Q)w_1)$. Hence we obtain by the usual iteration and continuation argument, 
\EQ{
 \pt \|\diff w^\pa\|_{\Str(-1,1)} \lec \|\diff w^\pa(0)\|_\HH,
 \pr \|\diff c^\pa\|_{L^\I(-1,1)} \lec \bde{\|\diff w^\pa(0)\|_\HH} \bde{\|w^\pa(0)\|_\HH},
 \pr \|\diff F(w^\pa,c^\pa)\|_{L^1_t\HH_x(-1,1)} \lec \|\diff w^\pa(0)\|_\HH. }

Next we prove the second and third estimates in \eqref{inc est}. Since they are now obvious for $\|v(0)\|_\HH\ge C_0\de$, we may assume that $\|v(0)\|_\HH\le C_0\de$. For the $v_d$ part, we have by the energy inequality 
\EQ{ \label{est diffree}
 \|v_d-e^{J\LL t}v_d(0)\|_{L^\I_t E(0,T)} \le \|D^{-2}M_\de(v)\|_{L^1_t \HH_x(0,T)} \lec \|v\|_{\Str(0,T)}^2 \lec \|v(0)\|_\HH^2.}
For the $\ga$ part, we have 
\EQ{
 \p_t\LR{\LL\ga|\ga}=\chi_\de(v)[A(v)\LR{f(Q)\ga_1|\na\ga_1}+A(v)\LR{\LL\na v_d|\ga} + \LR{N(v)|\ga_2}],}
and so 
\EQ{
 [\|\ga\|_E^2]_0^T \pt\lec \|(A(v)f'(Q)\ga_1,A(v)D^2 v_d,N(v))\|_{L^1_tL^2_x(0,T)}\|\ga\|_{L^\I_t\HH(0,T)} 
 \pr\lec \|v(0)\|_\HH^3}
 and we are done. 
\end{proof}

Denote the nonlinear propagator for  equation \eqref{eq v} on $\HH$ by 
\EQ{ \label{def U}
 U(t):\HH \to \HH, \pq U(t)v(0)=v(t).}

\subsection{Smallness of nonlinearity}

The following estimate on the nonlinear term by the mobile distance will be the basis of all the succeeding arguments.
\begin{lem} \label{lem:diff v}
For any two solutions $v^j(t)=U(t)v^j(0)\in C(\R;\HH)$ of \eqref{eq v}, we have 
\EQ{
 \pt \sup_{|t|\le 1}\tm v^\pa (t) \lec \tm v^\pa (0),
 \pr \sup_{|t|\le 1}\|P_d[\diff v^\pa(t)-e^{J\LL t}\diff v^\pa(0)]\|_E + \left|[\tm P_\ga v^\pa ]_0^t\right| 
 \pn\lec \de\tm v^\pa (0),}
where the implicit constants are determined by $d$, $f$, $Q$ and $\ka$. 
\end{lem}
\begin{proof}
Without loss of generality, we may assume that for some $q(t)\in\R^d$ 
\EQ{ \label{def q}
 \tm v^\pa (t)^2 \pt= \|P_d \diff v^\pa (t)\|_E^2 + \|\ga^0(t)-\ta_{q(t)}\ga^1(t)\|_E^2
 +|q(t)|^2\phi_\de(\|\ga^1(t)\|_E)^2}
for $|t|\le 1$. Decompose each solution by
\EQ{
 v^j = \sum_{\pm, k\in K}\la_{k\pm}^jg_{k\pm} + \mu^j\cdot\na\vec Q + \nu^j\cdot J\na\vec Q+ \ga^j.}

{\bf (I) Case $\|v^\pa (0)\|_2\le C_1\de$:} 
The previous lemma implies that $\|v^\pa (t)\|_2 \ll C_2\de$, and so $\phi_\de(\|\ga^j(t)\|_E)=1$, for $|t|\le 1$ and $j=0,1$. The discrete components solve
\EQ{
 \pt \p_t\diff\la_{k\pm}^\pa = \pm k \diff\la_{k\pm}^\pa + P_{\pm k}\diff{M_\de(v^\pa)},
 \pr \p_t\diff\mu^\pa = -\diff\nu^\pa + P_\mu\diff{M_\de(v^\pa)}, 
 \pq \p_t\diff\nu^\pa = P_\nu\diff{M_\de(v^\pa)},}
where the nonlinear term is bounded by
\EQ{ \label{est diffM}
 \|\diff{M_\de(v^\pa)}\|_{H^{-2}} \lec \de\tm v^\pa,}
which is proved as follows: \eqref{est on mob} and the translation invariance of $\om(v,\na v)$ imply 
\EQ{ \label{est diff H-2}
 \pt |\diff\chi_\de^\pa(v)|\lec |\diff \|v^\pa\|_2|/\de \lec \tm v^\pa /\de, 
 \pr |\diff A(v^\pa)| \lec \de[\|D^{-1}\diff v^\pa\|_{\HH}+\|\ga^0-\ta_q\ga^1\|_{\HH}  ] \lec \de\tm v^\pa ,
 \pr \|\diff\na v^\pa\|_{H^{-2}} \lec \|D^{-1}\diff v^\pa\|_{\HH} \lec \tm v^\pa .}
The nonlinear part is estimated by using Sobolev
\EQ{
 \|\diff N(v^\pa)\|_{H^{-2}} 
 \pt\lec \|N(v^0)-\ta_qN(v^1)\|_{L^\vr}+ \int_0^1  \|\p_\te \ta_{\te q}N(v^1)\|_{H^{-1}_\vr} d\te
 \pr\lec \de\|v^0_1-\ta_qv^1_1\|_{H^1} + |q|\|v^1_1\|_{H^1}^2
 \pn\lec \de\tm v^\pa,}
where $\vr:=\min(2,1+1/p)$ for $d\ge 2$ and $\vr:=1$ for $d=1$. 
Thus we obtain \eqref{est diffM}. Therefore, we have for $|t|\le 1$, 
\EQ{ \label{diff d err est}
 \|\diff v_d^\pa(t)-e^{J\LL t}\diff v_d^\pa(0)\|_E \lec \Bigl|\int_0^t \de\tm v^\pa(s) ds\Bigr|.}
The linearized solution enjoys the obvious bound 
\EQ{ \label{diff d free est}
 \|e^{J\LL t}\diff v_d^\pa(0)\|_E \lec e^{\ok t}\|\diff v_d^\pa(0)\|_E.}

For the difference in the $\ga$ component, we need the mobile distance. Let 
\EQ{ \label{def zej}
 \ta^j:=\ta_{c^j}, \pq 
 \z^j:=\ta^j\ga^j, \pq Q^j:=\ta^jQ=Q(\cdot-c^j),} 
for $j=0,1$, where $c^j(t)\in\R^d$ are the solutions of 
\EQ{
 \dot c^j=\chi_\de(v^j)A(v^j), \pq  c^0(0)=0, \pq c^1(0)=q(0),}
where $q$ has been chosen in \eqref{def q}. Then we have 
\EQ{
 \dot \z^j = J\LL^j\z^j + \dot c^j\cdot(P_\ga^j\na P_d^j \z^j - P_d^j\na \z^j)+P_\ga^j(0,M^j),}
where $\LL^j$, $P_d^j$ and $P_\ga^j$ are linear operators, and $M^j$ is the nonlinear part, defined by 
\EQ{
 \pt \LL^j:=\ta^j \LL(\ta^j)^*, \pq P_\star^j:=\ta^j P_\star(\ta^j)^* ,  
 \pq M^j:= \chi_\de(v^j)N_{c^j}(z^j),}
and $z^j:=\ta^jv_d^j+\z^j$. Hence the difference satisfies 
\EQ{
 \pt \diff{\dot \z^\pa} = J\LL^0 \diff\z^\pa + P_\ga^0(0,\diff M^\pa) + R,
 \pq \|R\|_\HH \lec \de[\tm v^\pa + |\diff c^\pa| + \|\diff \z^\pa\|_2],}
using \eqref{est diff H-2} as well as $\|v^j\|_\HH\lec\de$. 
The Strichartz estimate for the free Klein-Gordon equation yields (regarding $J(\LL^0-\D)\diff\z^\pa$ as a perturbation, which can
be done by partitioning the time-interval)
\EQ{ \label{est diff z hD}
 \pt \|\diff{\z^\pa}\|_{\Str(0,1)} \lec \|\diff\z^\pa(0)\|_2 + \|\diff M^\pa\|_{L^1_tL^2_x(0,1)}+\|R\|_{L^\I_t\HH_x(0,1)},}
where the nonlinear part is estimated by applying H\"older to the second order Taylor expansion of $f$:
\EQ{
 \pt f(Q^j+z^j_1)-f(Q^j)-f'(Q^j)z^j_1
 =\int_0^1\int_0^1[f''(Q^j+\al\te z^j_1)\te(z^j_1)^2]\,d\al d\te,}
and
\EQ{
 \pt\diff[f(Q^\pa+z^\pa_1)-f(Q^\pa)-f'(Q^\pa)z^\pa_1]
 \pr=\int_0^1\int_0^1[f''(Q^\al+\te z^\al_1)\diff(Q^\pa+\te z^\pa_1)+f''(Q^0+\al\te z^0_1)\te z^0_1\diff z^\pa_1]\,d\al d\te,}
where $Q^\al:=(1-\al)Q^0+\al Q^1$ and $z^\al:=(1-\al)z^0+\al z^1$. Thus we get 
\EQ{ \label{diff M est}
 \|\diff M^\pa\|_{L^1_tL^2_x(0,1)} \pt\lec (1+\|z^\pa\|_{\Str(0,1)})^{p-2}\|z^\pa\|_{\Str(0,1)}
 [\|\diff c^\pa\|_{L^\I_t(0,1)}+\|\diff z^\pa\|_{\Str(0,1)}].}
On the other hand, we have 
\EQ{ \label{diff c mov est}
 \pt \|\diff z^\pa\|_{\Str(0,1)} \lec \|\diff c^\pa\|_{L^\I(0,1)}+\|\diff v_d^\pa\|_{L^\I\HH(0,1)}+\|\diff\z^\pa\|_{\Str(0,1)}, 
 \pr  \tm v^\pa \lec \|\diff v_d^\pa\|_\HH + \|\diff\z^\pa\|_\HH + |\diff c^\pa|.}
Combining these estimates with \eqref{est diff z hD}--\eqref{diff M est}, \eqref{diff d err est} and \eqref{diff d free est}, we obtain
\EQ{ \label{mov bd}
 \sup_{0\le t\le 1}\tm v^\pa \lec \|\diff c^\pa\|_{C^1_t(0,1)}+\|\diff\z^\pa\|_{\Str(0,1)}+\|\diff v_d^\pa\|_{\Str(0,1)} \lec \tm v^\pa(0).}

For the sharper estimate on the $\ga$ part, we use 
\EQ{ \label{inc mov}
 \pt (\tm \ga^\pa)^2 \le \|(\ta^0)^*\diff\z^\pa\|_E^2 + |\diff c^\pa|^2,
 \pr \|(\ta^0)^*\diff\z^\pa\|_E^2 = \LR{\LL^0 \diff\z^\pa|P_\ga^0\diff\z^\pa}+\|P_d\ga^1(x+\diff c^\pa)\|_E^2,}
with equality at $t=0$. For the distance term, we have from \eqref{diff c mov est} and \eqref{mov bd} 
\EQ{ \label{inc c bd}
 |\diff c^\pa(t)| \le |\diff c^\pa(0)|+C\de\tm v^\pa(0).}
For the translated part, we have 
\EQ{
 \pt |\LR{\LL^0 \diff\z^\pa|P_\ga^0\diff\z^\pa}-\LR{\LL^0\diff\z^\pa|\diff\z^\pa}| \lec\|\diff P_d^\pa \z^1\|_2^2 \lec|\diff c^\pa|^2\de^2,
 \pr  \|P_d\ga^1(x+\diff c^\pa)\|_2^2 \lec |\diff c^\pa|^2\de^2,}
and 
\EQ{
 \p_t\LR{\LL^0\diff\z^\pa|\diff\z^\pa}=-\LR{f'(Q_{c^0})\dot c^0\cdot\na Q_{c^0}\diff\z^\pa|\diff\z^\pa}+2\LR{\LL^0\diff\z^\pa|P_\ga^0\diff M^\pa+R}. }
Hence for $|t|\le 1$, using \eqref{mov bd} and \eqref{diff M est} as well, 
\EQ{ 
 \pt\left|[\LR{\LL^0\diff\z^\pa|P_\ga^0\diff\z^\pa}]_0^t\right|
 \pr\lec |\diff c^\pa|^2\de^2+\de^2\|\diff\z^\pa\|_{L^\I_t \HH}^2+\|\diff\z^\pa\|_{L^\I_t \HH}\|P_\ga^0\diff M^\pa+R\|_{L^1_t \HH}
 \pn\lec \de^2(\tm v^\pa(0))^2.}
Plugging this and \eqref{inc c bd} into \eqref{inc mov}, we obtain the desired upper estimate on the $\ga$ part. For the lower estimate
one reverses time,
completing the proof in the case (I).

{\bf (II) Case $\min_j\|v^j(0)\|_2\ge C_0\de$}: The previous lemma implies that $\|v^j(t)\|_\HH\gec C_0\de\gg \de$, and so, for $|t|\le 1$ and $j=0,1$,  
\EQ{
 \pt v^j(t)=e^{J\LL t}v^j(0), \pq \|\ga^j(t)\|_E=\|\ga^j(0)\|_E, 
 \pq \|P_d\diff v^\pa(t)\|_E\lec\|P_d\diff v^\pa(0)\|_E.}
To estimate the $\ga$ part, let $\z(t):=\ga^0(t)-\ga^1(t,x-q(0))$. Then 
\EQ{
 \dot\z\pt-J\LL\z = (0, [f'(Q)-f'(Q_{q(0)})]\ta_{q(0)}\ga_1^1),}
where the right-hand side is bounded in $\HH$ by 
\EQ{
 |q(0)|\|\ga^1(0)\|_\HH \lec \de\tm v^\pa(0),}
where we used that $a\lec\de\phi_\de(a)$. 
Hence using the energy inequality for $\LL$, we get 
\EQ{
 \pt \left|[\LR{\LL\z|\z}]_0^t\right| \lec \de\tm v^\pa(0) \|\z\|_{L^\I_t\HH_x}, 
 \pr \|P_d\z(t)\|_E = \|P_d\ta_{q(0)}\ga^1\|_E \lec |q(0)|\|\ga^1(0)\|_\HH \lec \de\tm v^\pa(0),}
and so 
\EQ{\label{3 48}
 (\tm \ga^\pa(t))^2 \pt\le \|\z(t)\|_E^2+|q(0)|^2\phi_\de(\|\ga^1(0)\|_E)^2 
 \pr\le (\tm \ga^\pa(0))^2+C\de^2(\tm v^\pa(0))^2.}
This completes the proof in the case (II). 

{\bf (III) Case $\|v^0(0)\|_\HH>C_1\de\gg C_0\de>\|v^1(0)\|_\HH$}: The previous lemma implies that $\|v^0(t)\|_\HH\gec C_1\de \gg C_0\de\gec \|v^1(t)\|_\HH$ for $|t|\le 1$, and so by \eqref{est on mob}, 
\EQ{
 \tm v^\pa(t) \simeq \|v^0(t)\|_\HH \simeq \|v^0(0)\|_\HH \simeq \tm v^\pa(0) \gec \de.}
For the difference from the linearized solution, \eqref{est diffree} yields the desired estimate. 
The estimate on the increment of the $\ga$ part is similar to the case (I), but now $\z^0=\ga^0$ evolves linearly, which means that the nonlinear terms in $\diff\z^\pa$ depends only on $\z^1$. Hence \eqref{diff M est} is replaced with 
\EQ{
 \|\diff M^\pa\|_{L^1_tL^2_x}=\|M^1\|_{L^1_tL^2_x} \lec \|v^1\|_{\Str}^2 \lec \de^2.}
Noting that $\de\lec\tm(v^0(0),v^1(0))$, the rest of the argument goes through as in case~(I) above. 

{\bf (IV) Case $\|v^0(0)\|<C_0\de\ll C_1\de<\|v^1(0)\|_\HH$}: Although this is symmetric with the previous case, we have to check the mobile distance part, 
since there we introduced asymmetry with~\eqref{def q}. The difference appears in \eqref{inc mov}: 
\EQ{
 (\tm v^\pa)^2 \le \|(\ta^0)^*\diff\z^\pa\|_E^2 + |\diff c^\pa|^2\phi_\de(\|\ga^1\|_E)^2.}
However this is admissible, because 
\EQ{
 [|\diff c^\pa|\phi_\de(\|\ga^1\|_E)]_0^t \lec \de^2\phi_\de(\|v^1(0)\|_E) \simeq \de\tm v^\pa(0),}
and the remaining argument is the same as in the previous case. 
\end{proof}

\subsection{Evolution of graphs, center-stable case}
Now we consider the graphs of $v_{\ge 0}\mapsto v_-$ satisfying a Lipschitz condition. It is convenient to extend them to the whole $\HH$. Our class of graphs for the contraction argument is given by 
\begin{multline} \label{def Gelde}
 \G_{\ell,\de}:=\{G: \HH \to P_-\HH \mid G=G\circ P_{\ge 0},\ G(0)=0,\ \|\diff G(v^\pa)\|_E \le \ell\tm v^\pa\}
\end{multline}
for small $\ell>0$, and the graph of $G\in\G_{\ell,\de}$ is denoted by 
\EQ{ \label{def graph}
 \gr{G} := \{\fy\in\HH \mid P_-\fy=G(\fy)\}.} 
A center-unstable manifold will be found as the unique invariant graph, by the contraction mapping principle in $\G_{\ell,\de}$. 

For $p>d/(d-2)$, the Sobolev inequality does not imply that $\dot v$ is bounded in $L^2_x$, and consequently we can not prove strict invariance of $\G_{\ell,\de}$ for $t>0$, but the ``almost invariance" given below is sufficient for the contraction argument. 
\begin{lem} \label{lem:diff inv}
There exists $C_L\ge 1$ such that if $\ell,\de>0$ satisfy 
\EQ{ \label{cond el-de}
 \ok \ell^2 + \de \ll \uk, \pq \de \ll \ell \uk,}
then for any two solutions $v^j(t)=U(t)v^j(0)$ ($j=0,1$) satisfying 
\EQ{
 \|\diff v_-^\pa(0)\|_E \le \ell\tm v^\pa(0),}
one has 
\EQ{ \label{diff est}
 \|\diff v_-^\pa(t)\|_E \le \CAS{C_L \ell\tm v^\pa(t) &(|t|\le 1),\\ 
  \ell\tm v^\pa(t) &(1/2\le t\le 1).}}
\end{lem}
\begin{proof}
The linearized solutions of the discrete modes are estimated by 
\EQ{ \label{est linearized}
 \pt \min(e^{\pm\uk t},e^{\pm\ok t})\|P_\pm \fy\|_E \le \|P_\pm e^{J\LL t}\fy\|_E 
  \le \max(e^{\pm\uk t},e^{\pm\ok t})\|P_\pm\fy\|_E,
 \pr e^{-\ka|t|}\|P_0\fy\|_E \le \|P_0e^{J\LL t}\fy\|_E \le e^{\ka|t|}\|P_0\fy\|_E.}
The previous lemma implies that 
\EQ{ \label{est on diff v-}
 \|\diff v_-^\pa(t)\|_E \pt\le \|\diff e^{J\LL t}v_-^\pa(0)\|_E+C\de\tm v^\pa(0)
      \pr\le [\max(e^{-\uk t},e^{-\ok t})\ell+C\de]\tm v^\pa(0),}
and also,  
\EQ{
 \tm v^\pa(t)^2 \ge \|e^{J\LL t}\diff v_d^\pa(0)\|_E^2 + \tm \ga^\pa(0)^2 - C\de^{2}\tm v^\pa(0)^2.}
Plugging \eqref{est linearized} into the last estimate, we obtain
\EQ{
 \tm v^\pa(t)^2 \ge \CAS{[e^{-2\ka t}(1-\ell^2)+e^{-2\ok t}\ell^2 - C\de^{2}]\tm v^\pa(0)^2 &(0\le t\le 1),\\ [e^{-2\ok|t|}-C\de^{2}]\tm v^\pa(0)^2 &(|t|\le 1).}}
Combining it with \eqref{est on diff v-} yields for $|t|\le 1$,
\EQ{
 \|\diff v^\pa_{-}(t)\|_E \le (e^{\ok}\ell +C\de)(e^{\ok}+C\de)\tm v^\pa(t)
 \le 2\ell e^{2\ok}\tm v^\pa(t),}
provided that $\de\ll\ell\ll 1$. For $1/2\le t\le 1$, we obtain
\EQ{
 \|\diff v^\pa_{-}(t)\|_E \pt\le (e^{-\uk/2}\ell+C\de)(e^{-2\ka}(1-\ell^2)+e^{-2\ok}\ell^2-C\de^{2})^{-1/2}\tm v^\pa
 \pr\le(1-\uk/3+C\de/\ell)\ell[1+C(\de+\ka+\ell^2\uk)]\tm v^\pa \le \ell \tm v^\pa,}
under the condition \eqref{cond el-de} and $\ka\ll\uk\le 1$. 
\end{proof}

As an immediate consequence of the above lemma together with a mapping degree argument, we obtain the following result. 
\begin{lem} \label{lem:U on graph}
Under the condition \eqref{cond el-de}, $U(t)$ for $|t|\le 1$ defines a map $\U(t):\G_{\ell,\de}\to\G_{C_L\ell,\de}$ uniquely by the relation $U(t)\gr{G}=\gr{\U(t)G}$. 
Moreover, if $1/2\le t\le 1$, then $\U(t)$ maps $\G_{\ell,\de}$ into itself. 
\end{lem}
\begin{proof}
The previous lemma yields for any $\fy^0,\fy^1\in U(t)\gr{G}$, 
\EQ{
 \|\diff\fy_-^\pa\|_E \lec \ell \tm \fy^\pa.}
Since $\ell\ll 1$, it implies $\|\diff\fy_-^\pa\|_E \ll \tm \fy_{\ge 0}^\pa$. Then the conditions $U(t)\gr{G}\subset\gr{\U(t)G}$ and $\U(t)G\circ P_{\ge 0}=\U(t)G$ define $\U(t)G$ uniquely and consistently on the set
\EQ{
 P_{\ge 0}U(t)\gr{G} + P_-\HH.} 
The proof is complete once the above is shown to be $\HH$, for which we use the degree argument. Suppose for contradiction that there exists 
$\psi\in P_{\ge 0} \HH\setminus  P_{\ge 0}U(t)\gr{G}$. In other words, for any $a\in P_-\HH$, 
\EQ{
 U(-t)(a+\psi) \not\in \gr{G}.}
Let $m(a):=P_-U(-t)(a+\psi)-G(U(-t)(a+\psi))$, then $m$ is a continuous map from $P_-\HH$ to itself, such that $0\not\in m(P_-\HH)$. On the other hand, if $|a|\gg \de$, then 
\EQ{
 m(a)=e^{-J\LL t}a-G(e^{-J\LL t}\psi).}
Define $\Phi:\R^K\to P_-\HH$, $\Psi:\R^K\setminus\{0\}\to S^{K-1}$ and $\ti m:(0,\I)\times S^{K-1}\to S^{K-1}$ by 
\EQ{
 \pt \Phi(X) = \sum_{k\in K}X_k g_{k-}, \pq \Psi(X)=\frac{X}{|X|},
 \pq \ti m(R,\te) = \Psi\circ\Phi^{-1}\circ m\circ \Phi(R\te).}
Then $\ti m$ is continuous, but the degree of $\ti m(R,\cdot)$ is $0$ for small $R>0$ and $1$ for large $R$, which is a contradiction. 
Hence $\U(t)G$ is well-defined as a map on $\HH$ which is right-invariant for $P_{\ge 0}$. The Lipschitz bounds are immediate from the previous lemma. 
\end{proof}

\subsection{Contraction of graphs, center-stable case}
We introduce the following norm in $\G:=\bigcup_{\ell>0}\G_{\ell,\de}$
\EQ{ \label{def Gnorm}
 \|G\|_\G := \sup_{\psi\in\HH} \frac{\|G(\psi)\|_E}{\|\psi\|_E}.}
It is easy to see that the set $\G$ is independent of $\de>0$. For any $G\in\G_{\ell,\de}$ and any $\psi\in \HH$, we have $\|G(\psi)\|_E \le \ell\tm(\psi,0) \le \ell\|\psi\|_E$, and so 
\EQ{
 \|G\|_\G \le \ell.}
$\G$ is a Banach space with this norm, where each $\G_{\ell,\de}$ is a bounded closed set. The contraction argument is completed by
\begin{lem} \label{lem:contract}
In addition to \eqref{cond el-de}, let 
\EQ{ \label{cond2 el-de}
  \de  \ll \uk^2.}
Then the map $\U(t)$ is a contraction on $\G_{\ell,\de}$ for all $t\ge 1/2$. 
\end{lem}
\begin{proof}
Let $T\in[1/2,1]$. For any $G^j\in\G_{\ell,\de}$ for $j=0,1$ and any $\psi\in\HH$, let 
\EQ{
 v^j(t):=U(t-T)[P_{\ge 0}\psi+(\U(T)G^j)\psi].}
Since $P_{\ge 0}v^0(T)=P_{\ge 0}v^1(T)$, we have 
\EQ{
 \tm v^\pa(T)=\|\diff v_-^\pa(T)\|_E.}
Applying Lemma \ref{lem:diff v} from $t=T$, we get for $0\le t\le T$, 
\EQ{
  \|\diff v_d^\pa(t)-e^{J\LL (t-T)}\diff v_-^\pa(T)\|_E + \tm \ga^\pa(t) \lec \de\|\diff v_-^\pa(T)\|_E.}
Hence using the same estimate on the linearized solution as in \eqref{est linearized}
\EQ{ \label{back extend}
 \pt \|\diff v_-^\pa(0)\|_E \ge \|e^{-J\LL T}\diff v_-^\pa(T)\|_E-C\de\|\diff v_-^\pa(T)\|_E
 \ge (e^{\uk T}-C\de)\|\diff v_-^\pa(T)\|_E,
 \pr \tm v_{\ge 0}^\pa(0) \lec \de\|\diff v_-^\pa(T)\|_E.}
On the other hand, since $v_-^j(0)=G^j(v^j(0))$ and $G^j\in\G_{\ell,\de}$, 
\EQ{ \label{t=0 adjust}
 \|\diff v_-^\pa(0)\|_E \pt\le \|\diff G^\pa(v^0(0))\|_E+\|\diff G^1(v^\pa(0))\|_E
 \pr\le\|\diff G^\pa\|_{\G}\|v_{\ge 0}^0(0)\|_E+\ell\tm v_{\ge 0}^\pa(0).}
\eqref{inc est} as well as \eqref{est linearized} yields
\EQ{
 \|v_{\ge 0}^0(0)\|_E^2 \le (e^{2\ka T}+C\de)\|\psi\|_E^2. }
Inserting this and the second inequality of \eqref{back extend} 
into \eqref{t=0 adjust}, we obtain 
\EQ{
 \|\diff v_-^\pa(0)\|_E \le (e^{\ka T}+C\sqrt\de)\|\diff G^\pa\|_{\G}\|\psi\|_E + \ell\de\|\diff v_-^\pa(T)\|_E.}
Combining this and \eqref{back extend}, we conclude that 
\EQ{ \label{contraction} 
 \|\diff v_-^\pa(T)\|_E \pt\le (1-C\de\ell)^{-1}(e^{\uk T}-C\de)^{-1}(e^{\ka T}+C\sqrt\de)\|\diff G^\pa\|_{\G}\|\psi\|_E
 \pr\le e^{-(\uk-\ka)T}(1+C\sqrt\de)\|\diff G^\pa\|_{\G}\|\psi\|_E.}
\eqref{cond2 el-de} and $\ka\ll\uk$ imply that there exists  $\La<1$, determined by $\uk,\ka,\de,\ell$ such that 
\EQ{
 \frac{\|\diff v_-^\pa(T)\|_E}{\|\psi\|_E} \le \La\|\diff G^\pa\|_{\G}.}
Taking the supremum over all $\psi\in \HH$ yields 
\EQ{
 \|\diff \, \U(T)G^\pa\|_{\G} \le \La\|\diff G^\pa\|_{\G},}
as desired. The case $T>1$ is now obvious by iteration. 
\end{proof}
Thus we obtain 
\begin{thm} \label{mfd loc eq}
Suppose that $\ell,\de>0$ satisfy \eqref{cond el-de} and \eqref{cond2 el-de}. Then there exists a unique $G_*\in\G_{\ell,\de}$ such that $\U(t)G_*=G_*$ for all $t\ge 0$. The uniqueness holds for any fixed $t>0$. 
\end{thm}
\begin{proof}
For any $T\ge 1/2$, the above lemma implies that there is a unique fixed point of $\U(T)$ in $\G_{\ell,\de}$. Since the equation is invariant for time translation, it implies that $\U(t)G\in\G_{C_L\ell,\de}$ is also a fixed point for all $0\le t\le 1$. Then the uniqueness of the fixed point implies that $\U(t)G=G$ for all $0\le t\le 1$, and so for all $t\ge 0$. If $\U(t)H=H$ for some $t>0$ and some $H\in\G_{\ell,\de}$, then by iteration $\U(T)H=H$ for some $T\ge 1/2$, and so $H=G$. 
\end{proof}

Since $U(t)$ is invertible, $U(t)\gr{G_*} = \gr{G_*}$ for all $t\in\R$. 

\noindent The conditions \eqref{cond el-de} and \eqref{cond2 el-de} are satisfied for $\ell=O(\de)$ as $\de\to+0$, which implies that 
\EQ{
 \gr{G_*}\ni\fy,\ \|\fy\|_\HH\le\de \implies |\om(\fy,g_{k+})|\lec\de^2,}
in other words, $\gr{G_*}$ is normal at $0$ to $(-k,1)\rh_k$ for each $k\in K$.  

Notice that the above construction did not really use the special property of the generalized null space of the linearized operator. However, the constructed manifold makes sense for the original equation only on the subset $\HH_\perp$, for which we need the property that the generalized null space is exactly generated by the symmetries
 of the equation. 

\subsection{Evolution of graphs, unstable case}
We now carry out an analogous procedure for the finite-dimensional unstable manifold. 
Thus, we now consider the graphs of $v_{+}\mapsto v_{\le0}$ satisfying a Lipschitz condition 
\begin{multline} \label{def Gelde+}
 \G_{\ell,\de}^+:=\{G: \HH \to P_{\le0} \HH \mid G=G\circ P_{+},\ G(0)=0,\  \tm G(v^\pa) \le \ell \|\diff v^\pa_+ \|_E  \}
\end{multline}
for small $\ell>0$, and the graph of $G\in\G_{\ell,\de}^+$ is denoted by 
\EQ{ \label{def graph+}
 \gr{G} := \{\fy\in\HH \mid P_{\le0}\fy=G(\fy)\}.} 
The  unstable manifold will be found as the unique invariant graph, by the contraction mapping principle in $\G_{\ell,\de}^+$. 
We formulate the analogue of Lemma~\ref{lem:diff inv} in this case.

\begin{lem} \label{lem:diff inv+}
There exists $C_L\ge 1$ such that if $\ell,\de>0$ satisfy 
\eqref{cond el-de}, 
then for any two solutions $v^j(t)=U(t)v^j(0)$ ($j=0,1$) satisfying 
\EQ{
 \tm v^\pa_{\le 0} (0)    \le \ell   \|\diff v_+^\pa(0)\|_E   ,}
one has 
\EQ{ \label{diff est+}
 \tm v^\pa_{\le0}    (t)    \le \CAS{C_L \ell \|\diff v_+^\pa(t)\|_E   &(|t|\le 1),\\    
  \ell  \|\diff v_+^\pa(t)\|_E   &(1/2\le t\le 1).}}
\end{lem}
\begin{proof}
We again have \eqref{est linearized} 
for the  linearized solutions of the discrete modes. In particular, 
\EQ{
\|  \diff v^\pa_+(t) \|_E &\ge \min(e^{\uk t},e^{\ok t}) \| \diff v^\pa_+(0)\|_E - C\de \tm v^\pa(0)
}
By Lemma~\ref{lem:diff v}, for $t\ge0$, 
\EQ{
\tm v^\pa_{\le0} (t)^2 &\le \|\diff e^{J\LL t}v_{0-}^\pa(0)\|_E^2 + \tm \ga^\pa(0)^2 + C\de^2 \tm v^\pa(0)^2 \\
&\le (e^{2\kappa t} \ell^2  + C\de^2 (1+\ell^2) + \ell^2)  \| \diff v_+^\pa(0)\|_E^2
}
and one now concludes by combining these estimates, cf.~Lemma~\ref{lem:diff inv}. 
\end{proof}

One now has the following analogue of Lemma~\ref{lem:U on graph}. 

\begin{lem} \label{lem:U on graph+}
Under the condition \eqref{cond el-de}, $U(t)$ for $|t|\le 1$ defines a map $\U(t):\G_{\ell,\de}^+\to\G_{C_L\ell,\de}^+$ uniquely by the relation $U(t)\gr{G}=\gr{\U(t)G}$. 
Moreover, if $t\ge \f12$, then $\U(t)$ maps $\G_{\ell,\de}^+$ into itself. 
\end{lem}
\begin{proof} The mapping properties for $|t|\le 1$ and $\f12\le t\le 1$ are an immediate consequence of the previous lemma. The extension
to $t\ge\f12$ then follows by iteration. 
As in the case of the center-stable version, the main issue is to show that $P_+ U(t) \gr{G} = P_+\HH=\HH_+$ for all $|t|\le 1$. 
Thus take $\psi_0\in  \HH_+$ and denote the $R$-ball in $\HH_+$ by~$B_R^+$. Lemma~\ref{lem:diff v} implies that if $0<\ell\ll 1$ then 
\[
 \|P_+U(t)(\psi+G(\psi))\|_\HH \gtrsim R \qquad \forall \; |t| \le 1,\; \forall \psi \in \partial  B_R^+
\]
for any $G\in \G_{\ell,\de}^+$, and with absolute implicit constants.  This shows that, with $\Phi(\psi)=\psi+G(\psi)$,  
\[
\mathrm{deg} (P_+ U(t)\Phi , B_R^+, \psi_0)=1 \quad\forall\; |t|\le1
\]
provided $R$ is sufficiently large, 
and we are done. 
\end{proof}

\subsection{Contraction of graphs, unstable case}

Let $\G^+:=\bigcup_{\ell>0}\G_{\ell,\de}^+$. As before, the set $\G^+$ is independent of $\de>0$. 
We introduce the following quasi-distance $d_+$ in $\G^+$: for any $G^1, G^2\in \G^+$ let 
\EQ{ \label{def G+ norm}
  d_+(G^1, G^2)   := \sup_{\psi\in\HH} \frac{\tm G^\pa (\psi)}{\|\psi\|_E}.}
It is clear that this expression is finite, and that it satisfies a triangle inequality with the same multiplicative loss as in~Proposition~\ref{prop mb}: 
\EQ{\label{d+ C0}
d_+(G^1, G^3)\le C_d( d_+(G^1, G^2) + d_+(G^2, G^3) ).
}
Moreover, $\G^+$ is a complete quasi-distance space, in which $\G_{\ell,\de}^+$ is closed. Recall that the Banach fixed point theorem is valid in complete quasi-distance spaces: 
\begin{lem} \label{G+ d+} 
Let $X$ be a complete quasi-distance space, and let $A:X\to X$ be a contraction. Then there is a unique fixed point $x_*\in X$ of $A$, which is obtained by $x_*=\lim_{n\to\I}A^n(x)$ for any $x\in X$.
\end{lem} 
\begin{proof}
Let $C\ge 1$ be the constant in the quasi-triangle inequality in $X$, let $\La\in(0,1)$ be the Lipschitz constant of $A$, and fix $m\in\N$ so that $\La^mC<1$. Take any $x_0\in X$ and let $x_n=A^{nm}(x_0)$ for each $n\in\N$. Then
\EQ{
 d(x_{n+1},x_n)=d(A^m(x_n),A^m(x_{n-1}))\le\La^md(x_n,x_{n-1})
 \le\cdots\le \La^{mn}d(x_1,x_0)}
Hence for any $k>j\ge 1$, by repeated use of the quasi-triangle inequality, 
\EQ{
 d(x_k,x_j)\pt\le Cd(x_k,x_{j+1})+Cd(x_{j+1},x_j)
 \le\cdots\le \sum_{l=j+1}^{k}C^{l-j}d(x_l,x_{l-1})
 \pr\le\sum_{l=j+1}^{k}(C\La^m)^{l-j}\La^{mj}d(x_1,x_0)
 \le\frac{\La^{mj}}{1-C\La^m}d(x_1,x_0).}
Hence $x_n\to\exists x_\I\in X$, and by the continuity of $A$, $A^m(x_\I)=x_\I$. Then 
\EQ{
 d(A(x_\I),x_\I)=d(A^{m+1}(x_\I),A^m(x_\I))\le \La^m d(A(x_\I),x_\I),}
which implies that $A(x_\I)=x_\I$. The uniqueness follows in the well-known way. 
\end{proof}
The following is an analogue of Lemma \ref{lem:contract}, but here the evolution time has to be long enough to absorb the quasi-triangle factor $C_d$ in using the ``chain-rule" in $\G^+$.
\begin{lem}
There are $\de_0>0$ and $T>0$ such that if $\de\le\de_0$ and \eqref{cond el-de} is satisfied, then the map $\U(t)$ is a contraction on $\G_{\ell,\de}^+$ for all $t\ge T$.  
\end{lem}
\begin{proof} 
Let $G^1, G^2\in \G_{\ell,\de}^+$, $T>0$, $\psi\in\HH_+$, and 
\EQ{
 v^j(t):=U(t-T)[\psi+  (\U(T)G^j)\psi],\qquad (j=0,1).} 
Iterating Lemmas \ref{lem:bd v} and \ref{lem:diff v} from $t=T$ down to $t=0$, we obtain 
\EQ{
\pt \|v^\pa(t)\|_E \le Ce^{CT}\|\psi\|_E, 
\pq \tm v^\pa(t) \le Ce^{CT}\tm v^\pa(T)=Ce^{CT}\tm v_{\le 0}^\pa(T), 
}
for $0\le t\le T$ with some constant $C\ge 1$ (determined by $d$, $Q$, $f$ and $\ka$). Hence if 
\EQ{ \label{cond de-T}
 \de \ll e^{-2CT}/C^2,}
then by iteration of those lemmas again, we deduce that   
\EQ{ \label{long-time est}
 \pt \|v^0_+(0)\|_E \lec e^{-\uk T/2}\|\psi\|_E,
 \pr \tm v_{\le 0}^\pa(0) \ge e^{-2\ka T}\tm v_{\le 0}^\pa(T),
 \pq \|\diff v_+^\pa(0)\|_E \lec \sqrt\de \tm v_{\le 0}^\pa(T).}
Since $v_{\le 0}^j(0)=G^j(v^j(0))$ and $G^j\in\G_{\ell,\de}^+$, 
\EQ{ \label{chain est}
 \tm v_{\le 0}^\pa(0) \pt\le C_d[\tm G^\pa(v^0(0)) + \tm G^1(v^\pa(0))]
 \pr\le C_d[d_+(G^\pa)\|v_+^0(0)\|_E + \ell\|\diff v_+^\pa(0)\|_E].}
Plugging \eqref{long-time est} into the above, we obtain 
\EQ{
 \tm v_{\le 0}^\pa(T) \le e^{2\ka T}C_d[d_+(G^\pa)Ce^{-\uk T/2}\|\psi\|_E + \ell\sqrt\de\tm v_{\le 0}^\pa(T)]. }
Choosing $T$ so large while keeping \eqref{cond de-T}, we can ensure that 
\EQ{
 \tm (\U(T)G^\pa)\psi = \tm v_{\le 0}^\pa(T) \le \La\|\psi\|_E,}
for some constant $\La\in(0,1)$ determined by $d,f,Q,\ka,T$ and $\de$. 
Obviously, this remains to be true even if we replace $T$ with any $T'\in[T,2T]$, taking $\de$ even smaller if necessary. Hence iterating $\U(t)$ allows one to draw the same conclusion for all $t\ge T$. 
\end{proof}
 
 By the same arguments as in Theorem~\ref{mfd loc eq} one now concludes the following. 
 
 \begin{thm} \label{u mfd loc eq}
Suppose that $\ell,\de>0$ satisfy the assumptions of the previous lemma. Then there exists a unique $G_*^+\in\G_{\ell,\de}^+$ such that $\U(t)G_*^+=G_*^+$ for all $t\ge 0$. The uniqueness holds for any fixed $t>0$.  Moreover, if $v(0)\in \gr{G_*^+}$, then $\| U(t)v(0) \|_{\HH}\to0$ exponentially as $t\to-\I$; in fact, for any $\e>0$, 
\EQ{\label{uk eps}
e^{-(\uk-\e) t} \| U(t)v(0)\|_{\HH}\to0 \qquad t\to-\I
}
\end{thm}
\begin{proof}
The estimate \eqref{uk eps} follows from the previous proof. In fact, \eqref{long-time est} implies 
\EQ{
 \|v(t)\|_E \lec e^{\frac{\uk}{2} t}\|v(0)\|_E \pq(t\to\I),}
where $v(t):=U(t)v(0)$, but we could take the exponent arbitrarily close to $\uk$ by choosing $\de$ even smaller. Since $v(t)$ comes into any $\de$ ball as $t\to-\I$, we may apply such decay estimates for $t$ sufficiently close to $-\I$, thereby deducing \eqref{uk eps} for all $\e>0$. 
\end{proof}

\section{The unstable manifold}

We now describe the unstable manifold in the original $u$-formulation of the equation, see~\eqref{NLKG}. 
Let $G_*^+$ be as in Theorem~\ref{u mfd loc eq}. For any $v(0)\in \gr{\G_{\ell,\de}^+}$ with $\| v(0)\|_{\HH} < \de$, define $v(t)=U(t) v(0)$, 
\EQ{\label{u from v}
u(t) = (\vec Q+v)(t,\cdot-c(t)), \quad c(0)=c_0, \quad \dot c(t)=A(v(t))
}
where $c_0\in\R^d$ is a fixed vector. 
By construction, $u$ solves \eqref{NLKG}, and by~\eqref{uk eps} one has $\dot c(t)\to0$ and $c(t)\to c(-\I)$ exponentially fast
as $t\to-\I$. 
In particular,   $u$ has vanishing momentum: $P(u)=0$. Then by design (cf.~\eqref{def Hperp}--\eqref{eq1 v}), $\omega(v, J\nabla\vec Q)$ is constant, and since it converges to zero
as $t\to-\I$, must vanish.  To summarize, we have obtained the following characterization of the unstable manifold. 

\begin{cor}
The unstable manifold $\M_u$ is the set of all $\vec u(0)$ with $u$ defined in terms of $G_*^+$  by means of~\eqref{u from v}. 
$\M_u$ is invariant in backward time, and all solutions starting in $\M_u$ converge to a trajectory of the form $\vec Q(\cdot -c(t))$ exponentially
fast as $t\to-\I$, with $\dot c(t)\to0$ as $t\to-\I$ exponentially fast. $\M_u$ is a Lipschitz manifold of dimension $K+d$. 
\end{cor}

The dimension count is a result of the fact that $\gr{G_*^+}$ is of dimension~$K$, and the translations (see~$c_0$ in~\eqref{u from v}) add  another $d$
dimensions.

\section{Trapping property of the center-stable manifold}
\subsection{Restriction by the orthogonality}
For any Banach space $X$, denote the ball around $0$ of radius $R>0$ by 
\EQ{ \label{def BR}
 B_R(X):=\{\fy\in X\mid \|\fy\|<R\}.}
\begin{lem} \label{lem:tiG}
If $\ell\le 1$, and $\de>0$ is small enough (depending only on $d,f,Q$), then for any $G\in\G_{\ell,\de}$, there is a unique map $\ti G:P_{\ga+}B_\de(\HH)\to P_{0-}\HH$ such that 
\EQ{
 \gr{\ti G}:\pt=\{\psi+\ti G(\psi)\mid \psi\in P_{\ga+}B_\de(\HH)\}
 \pr=\{\fy\in\gr{G}\cap\HH_\perp \mid P_{\ga+}\fy\in P_{\ga+}B_\de(\HH), |P_\nu\fy|<\de\}.}
Moreover, $\ti G$ is Lipschitz continuous in the mobile distance $\tm$. 
\end{lem}
\begin{proof}
For any $\psi\in\HH$ and $\nu\in\R^d$, put 
\EQ{
 \ti\psi(\nu):=\psi+\nu\cdot J\na\vec Q, \pq \ti\fy(\nu):=\ti\psi(\nu)+G(\ti\psi(\nu)).} 
It suffices to show that for any $\psi\in P_{\ga+}B_\de(\HH)$, there is a unique fixed point $\nu\in B_\de(\R^d)$ for the map 
\EQ{
 \nu\mapsto \NN(\nu):=H(Q)^{-1}\om(\ti\fy(\nu),\na\ti\fy(\nu))/2.}
For any $\nu^j\in B_\de(\R^d)$ ($j=0,1$), we have 
\EQ{
 \diff{\ti\fy}(\nu^\pa)=\diff\nu^\pa\cdot J\na\vec Q+\diff G(\psi+\nu^\pa\cdot J\na\vec Q).}
Hence using $\|\psi\|_\HH<\de$ and $G\in\G_{\ell,\de}$ as well, we deduce that 
\EQ{
 |\NN(\nu)|\lec\|\ti\fy(\nu)\|_\HH^2 \lec \de^2, \pq
 |\diff\om(\ti\fy(\nu^\pa),\na\ti\fy(\nu^\pa))| \lec \de|\diff\nu^\pa|.}
Therefore $\NN:B_\de(\R^d)\to B_\de(\R^d)$ is a contraction for small $\de,\ell>0$, and so has a unique fixed point $\nu=\NN(\nu)\in B_\de(\R^d)$. 
Since $\NN$ is Lipschitz for $\psi$ in the mobile distance, so is the fixed point $\nu(\psi)$, as well as $\ti G$. 
\end{proof}
Hence $\gr{G}\cap\HH_\perp$ is a Lipschitz manifold in the mobile distance around $0$ with codimension $K+2d=\dim P_{0-}\HH$. 

\subsection{Solutions on the center-stable manifold with the orthogonality}
Let $G=G_*\in\G_{\ell,\de}$ be the map for the center-unstable manifold of the localized equation given by Theorem \ref{mfd loc eq}, and let $\ti G=\ti G_*$ be the map for its orthogonal restriction given by the above lemma. The invariance of $\gr{G}$ means that for any $v(0)\in\gr{G}$, $v(t):=U(t)v(0)$ stays on $\gr{G}$ for all $t\in\R$. Let $c(t)$ be the solution of 
\EQ{ \label{eq c}
 c(0)=0, \pq \dot c(t)=A(v(t)),}
and $u(t)=(\vec Q+v)(t,x-c(t))$. If $v(0)\in B_\de(\HH)$, then $u$ solves the original equation \eqref{NLKG} as long as $v(t)\in B_\de(\HH)$. Meanwhile, the momentum $P(u)=\om(v,\na\vec Q)+\om(v,\na v)/2$ is preserved, and so is $\om(v,J\na\vec Q)$, because of $\dot c=A(v)$. Hence if $v(0)\in\HH_\perp\cap B_\de(\HH)$, then $v(t)$ remains there as long as $v(t)\in B_\de(\HH)$. 

To see that the solution stays in the neighborhood for $t<0$, expand the conserved energy by $u=(\vec Q+v)(x-c)$
\EQ{ \label{energy expand}
 E(u)-J(Q) = -\sum_{k\in K}k\la_{k+}\la_{k-} + \frac 12\LR{\LL\ga|\ga} - C(v),}
where the nonlinear energy $C$ is defined by
\EQ{ \label{def C}
 C(v) := f(Q+v_1) - f(Q) - \LR{f'(Q)|v_1} - \frac 12\LR{f''(Q)v_1|v_1} = o(\|v\|_\HH^2).}

Suppose that for some $t_0<0$ 
\EQ{
 \|v(0)\|_\HH=\e\ll\de, 
 \pq \max_{t_0\le t\le 0}\|v(t)\|_\HH<\de, 
 \pq \|v(t_0)\|_\HH\gg\sqrt\ell\de+\e(\ll \de).} 
Then $|E(u)-J(Q)|\simeq\e^2$. Since $v(t)\in\gr{G}$ and $G\in\G_{\ell,\de}$, \eqref{energy expand} implies that 
\EQ{ \label{ga smallness}
 \LR{\LL\ga|\ga}\lec \sum_k  k |\la_{k+}\la_{k-} | + E(u)-J(Q) \lec \ell\de^2+\e^2 \ll \de^2,}
for $t_0\le t\le 0$, which together with $v(t)\in\HH_\perp$ implies
\EQ{
 \sum_k|\la_{k+}(t_0)|^2 \simeq \|v(t_0)\|_\HH^2.} 
Now consider the nonlinear energy functional 
\EQ{
 \E(v):\pt=E(u)-J(Q)+\sum_{k\in K}\frac{k}{2}(\la_{k+}+\la_{k-})^2
 \pr=\sum_{k\in K}\frac{k}{2}(\la_{k+}^2+\la_{k-}^2)+\frac12\LR{\LL\ga|\ga}-C(v).}
Since $v(t)\in\HH_\perp$, we have $\E(v)\simeq\|v\|_\HH^2$, and moreover, the equation of $\la_{k\pm}$, see~\eqref{eq v in spec}
 together with conservation of $E(u)$ yields 
\EQ{
 \frac{d}{dt}\E(v)=\sum_{k}k^2(\la_{k+}^2-\la_{k-}^2)+o(\|v\|_\HH^2|\la_{k\pm}|),}and so 
\EQ{
 \frac{d}{dt}\E(v(t_0)) \gec \sum_k k^2\la_{k+}^2 \simeq \E(v(t_0)).}
Therefore $\E(v(t))$ cannot increase beyond $O(\ell\de^2+\e^2)$ as $t<0$ decreases. 

In conclusion, for any $v(0)\in\gr{\ti G_*}$ such that $\|v(0)\|_\HH\ll\de$, the solution $v(t)=U(t)v(0)$ remains on $\gr{\ti G_*}$, $\|v(t)\|_\HH\ll\de$ for all $t<0$ and $u(t)=(\vec Q+v)(x-c)$ with $\dot c=A(v)$ solves the original equation \eqref{NLKG} for all $t<0$. Thus we obtain a center-unstable manifold of the original equation with zero total momentum. 

More precisely, fix $0<\de'\ll\de$ and let 
\EQ{
 \M_{cu,0}\pt:=\{(\vec Q+U(t)\fy)(x-c) \mid \fy\in\gr{\ti G_*},\ \|\fy\|_\HH<\de',\ t< 0,\ c\in\R^d\},}
then for any initial data $u(0)=(\vec Q+U(t)\fy)(x-c)\in\M_{cu,0}$, the solution $u(t)$ of \eqref{NLKG} is on $\M_{cu,0}$ for all $t\le 0$ and $P(u(t))=0$. Moreover, $U(t)\fy\in\gr{\ti G_*}$ and $\|U(t)\fy\|_\HH\lec\de'$ for all $t\le 0$. The nonlinear projection 
\EQ{
 P_\perp:\HH_0\ni u\mapsto v\in\HH_\perp;\pq u=(\vec Q+v)(x-c)}
is uniquely defined in a neighborhood of the translation family of stationary solutions 
\EQ{ \label{def S0Q}
 \Sm_0(Q):=\{\vec Q(x-q)\}_{q\in\R^d} \subset \HH_0,}
by solving the equation 
\EQ{
 0=\om(v,J\na\vec Q)=\om(u(x+c),J\na\vec Q)=\om(u,J\na\vec Q(x-c)).}
Indeed it can be solved locally by the implicit function theorem, since if $u=\vec Q(x-c_0)+\psi$, $\|\psi\|_\HH\lec\de$ then 
\EQ{
  \na_c\om(u,J\na\vec Q(x-c))
 \pt=-\om(\vec Q(x-c_0)+\psi,J\na^2\vec Q(x-c))
 \pr=H(Q)+O(|c-c_0|+\de).}
Since the mapping $u\mapsto c$ thereby defined is smooth, the map $u\mapsto(v,c)$ is (locally) bi-Lipschitz in the mobile distance from $\HH_0$ to $\HH_\perp\oplus\R^d$. Since $\M_{cu,0}$ is mapped onto a $0$-neighborhood of $\gr{\ti G_*}\oplus\R^d$, the codimension of $\M_{cu,0}$ in $\HH_0$ is equal to that of $\gr{\ti G_*}$ in $\HH_\perp$, which is $K$.   

Thus we have obtained a center-unstable manifold $\M_{cu,0}$ of $\Sm_0(Q)$ in $\HH_0$ with codimension $K$. Its time inversion 
\EQ{
 \overline{\M_{cu,0}}=:\M_{cs,0}\subset\HH_0} 
is a center-stable manifold of $\Sm_0(Q)$.  

\subsection{Lorentz extension of the center-stable manifold}
Using the Lorentz transform 
\EQ{
 u(t,x) \mapsto u_p:=u(\LR{p}t+p\cdot x,x+p(\LR{p}-1)|p|^{-2}p\cdot x+tp) \pq (p\in\R^d),}
we can further extend $\M_{cs,0}$ to a manifold $\M_{cs}$ of codimension $K$, around the soliton manifold $\Sm(Q)$. Indeed, \eqref{NLKG} is invariant for any Lorentz transform, while the total energy and momentum are transformed 
\EQ{
 \pt E(u_p) = E(u)\LR{p}+P(u)\cdot p, 
 \pq P(u_p) = P(u)\LR{p}+E(u)p.}
$E^2-|P|^2$ is invariant, which is positive around $\Sm(Q)$. Hence there is a unique $p\in\R^d$ for each solution $u$ near the traveling waves, such that $P(u_p)=0$ and $E(u_p)=\sqrt{E(u)^2-|P(u)|^2}$. 

However, one needs to be more careful because the Lorentz transform mixes space-time and a solution from $M_{cs,0}$ may not be global in the negative time. Indeed, from~\cite{NLKGrad,NLKGnonrad} we know that ``half'' of the solutions on $\M_{cs,0}$ (namely, as given by the separating surface~$\M_{cu,0}$) blow up in  negative time, at least when $\vec Q$ is the ground state and $f(u)=|u|^{p+1}$, $p>1+4/d$. 

The local wellposedness implies that for any $T>0$ there is $\de>0$ such that for any initial data within distance $\de$ from $\Sm_0(Q)$, the solution extends at 
least for times $|t|<T$. The exponential decay of $Q$ implies that for any $\de>0$ there is $R>0$ such that for any such initial data, the free energy in the exterior region $|x-q|>R$ is less than $O(\de^2)$ for some $q\in\R^d$. The local wellposedness, the conservation of the energy and the Sobolev inequality  implies that every solution with small initial free energy is global, keeping the same size of free energy for all time. Hence the finite speed of propagation of the free Klein-Gordon equation implies that for small $\de>0$, every solution with  free energy $O(\de^2)$ in $|x-q|>R$ at $t=0$ is extended to the whole exterior cone $|x-q|>R+|t|$ with the same size of the free energy on any time slice of it. 

Thus in conclusion, there is $R>0$ and $\de(T)>0$ for any $T>0$ such that every solution starting on $\M_{cs,0}$ and 
within distance $\de$ from $\Sm_0(Q)$ is extended to the space-time region 
\EQ{
 \{(t,x)\in\R^{1+d} \mid t>-T \text{ or } |x-q|>R+|t|\},}
for some $q\in\R^d$. For any Lorentz transform~$L$, there is $T>0$ such that the image of the above region under~$L$ contains $\{(t,x)\mid t\ge 0\}$. In other words, the image of any solution on $\M_{cs,0}$ close to $\Sm_0(Q)$ is extended to a forward global solution. The invariance of the solution set of $\M_{cs,0}$ for the space and backward time translations is also inherited by the image, because such a translation of the Lorentz transform is the Lorentz transform of another translation. 
It is also easy to see that these  solution remains close to the corresponding traveling wave. 

However, it seems difficult to make the above argument uniform with respect to the Lorentz transform: the larger the momentum~$p$,  the smaller the neighborhood of $\Sm_0(Q)$ needs to be chosen. This is why the resulting manifold is not strictly Lorentz invariant, but only within a neighborhood of $\Sm(Q)$ depending on the Lorentz transform (but the neighborhood can be chosen uniformly for $p$ in compact sets). 

Thus we obtain a center-stable manifold $\M_{cs}$ of the soliton manifold $\Sm(Q)$. $\M_{cs}$ can be identified with the set of forward global solutions starting from it, where each solution is characterized uniquely by the total momentum and its Lorentz transform with $0$-momentum starting from $\M_{cs,0}$. In this way\footnote{The general solution close to $\Sm(Q)$ may well blow up in both time directions, but the smaller neighborhood yields the bigger lower bound on the existence time, which is sufficient for the construction of this bi-Lipschitz map.}, 
we can define a bi-Lipschitz map from a neighborhood of $\Sm(Q)$ in $\HH$ to a neighborhood of $\Sm_0(Q)\oplus\R^d$ in $\HH_0\oplus\R^d$. Since it maps $\M_{cs}$ onto the intersection of $\M_{cs,0}\oplus\R^d$ with a neighborhood of $\Sm_0(Q)\oplus\R^d$, the codimension of $\M_{cs}$ in $\HH$ is also $K$.

\subsection{Solutions off the center-stable manifold}
It remains to describe the dynamics off the manifold, or more specifically, the {\em repulsive
property} of the center-unstable manifold in negative time. 
For this we need some  sort of opposite to Lemma \ref{lem:diff inv}: 
\begin{lem}
If $\ell,\de>0$ satisfy 
\EQ{ \label{cond3 el-de}
 \de \ll \ell \uk,}
then for any two solutions $v^j=U(t)v^j(0)$ ($j=0,1$) satisfying 
\EQ{
 \max(\|\diff v_{0+}^\pa(0)\|_E,\tm \ga^\pa(0)) \le \ell\|\diff v_-^\pa(0)\|_E,}
one has 
\EQ{ 
 \max(\|\diff v_{0+}^\pa(t)\|_E,\tm \ga^\pa(t)) \le \CAS{2\ell\|\diff v_-^\pa(t)\|_E &(-1/2< t\le 0),\\ 
  \ell\|\diff v_-^\pa(t)\|_E &(-1\le t\le -1/2),}}
and 
\EQ{
 \|\diff v_-^\pa(t)\|_E \ge \CAS{ \frac 12e^{-\uk t/2}\|\diff v_-^\pa(0)\|_E &(-1/2<t\le 0), \\ e^{-\uk t/2}\|\diff v_-^\pa(0)\|_E &(-1\le t\le -1/2).}} 
\end{lem}
\begin{proof}
Let $\ti m(0):=\|\diff v_-^\pa(0)\|_E\simeq\tm v^\pa(0)$. Lemma \ref{lem:diff v} implies that 
\EQ{
 \pt\|\diff v_{0+}^\pa(t)\|_E  \le (e^{-\ka t}\ell + C\de)\ti m(0),
 \pq\|\diff v_-^\pa(t)\|_E \ge (e^{-\uk t}-C\de)\ti m(0), 
 \pr\tm \ga^\pa(t) \le (\ell + C\de)\ti m(0), }
for $-1\le t\le 0$. Hence
\EQ{
 \max(\|\diff v_{0+}^\pa(t)\|_E,\tm \ga^\pa(t))
 \pt\le (e^{-\ka t}\ell + C\de)(e^{-\uk t}-C\de)^{-1}\|\diff v_-^\pa(t)\|_E
 \pr\le e^{(\uk-\ka)t}(\ell+C\de)\|\diff v_-^\pa(t)\|_E,}
and the conclusion follows from \eqref{cond3 el-de} as well as $\ka\ll\uk$. 
\end{proof}

Let $v^0(0)\in B_\de(\HH)\cap\HH_\perp\setminus\gr{\ti G_*}$, and 
\EQ{
 \psi:=P_{\ga+}v^0(0), \pq v^1(0):=\psi+\ti G_*(\psi).}
Then we have
\EQ{
 \diff v_{\ga+}^\pa(0)=0, \pq \|\diff v_0^\pa(0)\|_E^{2}\lec\|\diff v^\pa(0)\|_E^2\simeq\|\diff v_-^\pa(0)\|_E^2\lec\de^2.}
Hence we can repeatedly apply the above lemma to deduce that 
\EQ{
 \|\diff v_-^\pa(t)\|_E \ge \frac 12e^{-\uk t/2}\|\diff v_-^\pa(0)\|_E}
for all $t<0$. In particular,
\EQ{
 \|v_-^0(t)\|_E \ge \|\diff v_-^\pa(t)\|_E-\|v_-^1(t)\|_E \gg \de}
for sufficiently large $-t$. 

In short, any solution starting from $\HH_\perp\cap B_\de(\HH)\setminus\gr{\ti G_*}$ 
moves out of the neighborhood $B_\de(\HH)$ for large $-t$. Of course, this is meaningful 
for the original equation only until the (backward) exiting time, but it implies that any trapped 
solution in $\HH_\perp$ within distance $\de$ must be on the manifold $\gr{\ti G_*}$ for large $-t$. 

Combining this with the result in the previous section, we conclude that the local center-unstable manifold $\M_{cu,0}$ is characterized as the collection of solutions with $0$-momentum which stay close to $\Sm_0(Q)$ for all $-t\ge 0$. By symmetry, $\M_{cs,0}$ is the collection of solutions with $0$-momentum which stay close to $\Sm_0(Q)$ for all $t\ge 0$. 

Let $\ti\M_{cu,0}$ be the maximal forward evolution of $\M_{cu,0}$, and let $\ti \M_{cs,0}$ be the maximal backward evolution of $\M_{cs,0}$. Then $\ti\M_{cs,0}$ is the collection of solutions that stay close to $\Sm_0(Q)$ for large $t$, namely the initial data set for which the solution will be trapped by $\Sm_0(Q)$. We have the same characterization for $\ti\M_{cu,0}$ for $t\to-\I$. By the Lorentz transform, we can extend them to solutions with nonzero momentum which are trapped by $\Sm(Q)$ with the same momentum. 

\section{Regularity of the center-stable manifold}
The above construction implies only Lipschitz continuity of the manifold. For the differential structure of $\gr{G_*}$, we also have to take account of the spatial translation. In the following, we assume that $f$ satisfies \eqref{asm f'} and $\al:=\max(1,p-2)$.  

\begin{defn} \label{def mob diff}
Let $Y$ be a Banach space. We say that a function $G:\HH\to Y$ is {\it mobile-differentiable} at $\fy\in \HH$, if there is a bounded linear $M:\HH\times\R^d\to Y$ such that 
\EQ{ \label{mob diff}
 \lim_{\e\to 0}\|G(\fy_{(\e)})-[G(\fy^0) + \e M(\psi,q)]\|/\e=0,}
where $\fy_{(\e)}:=(\fy+\e\psi)(x+\e q)$, for any $(\psi,q)\in \HH\times\R^d$. It is obvious that $M$ is unique. We call $\Dm G(\fy):=M$ the mobile derivative of $G$ at $\fy$. 
\end{defn}
Let $G'(\fy)$ be the usual derivative in the Frech\'et sense. Then we have 
\EQ{
 \Dm G(\fy)(\psi,q)=G'(\fy)(\psi+\na\fy\cdot q),}
provided that $G$ is differentiable in the $\D\HH$ topology, but in general, it makes sense only in the subspace $q=0$. Hence the mobile-differentiability is stronger than the differentiability in $\HH$, and weaker than that in $\D\HH$. 

If $G\in\G_{\ell,\de}$ and mobile-differentiable, then 
\EQ{
 \|G(\fy_{(\e)})-G(\fy^0)\|_\HH \lec \ell\mb(\fy_{(\e)},\fy^0)
 \lec \ell\e[\|\psi\|_\HH+|q|\phi_\de(\|\fy\|_\HH)],}
which implies  
\EQ{ \label{mobdif Lip bd}
 \|\Dm G(\fy)(\psi,q)\|_\HH \lec \ell[\|\psi\|_\HH+|q|\phi_\de(\|\fy\|_\HH)].}
Moreover, we have 
\EQ{
  G((\fy+\e\psi+o(\e))(x+\e q+o(\e))
 \pt=G(\fy_{(\e)})+o(\e) 
 \pr= G(\fy) + \e\Dm G(\fy)(\psi,q) + o(\e).}

We are going to prove that $G_*:\HH\to P_-\HH$ is ``mobile-$C^{1,\al}$", by showing the flow-invariance of the following set of such graphs. 

\begin{defn} \label{def mob C1al}
For each $\de,\ell,\La>0$, and $\al\in(0,1]$, 
we define $\G_{\ell,\de}^{\al,\La}$ as the set of all $G\in\G_{\ell,\de}$ that are mobile differentiable at every $\fy\in\HH$, satisfying
\EQ{ \label{mob C11}
 \pt \|\Dm G(\fy^0)(\psi,q)-\Dm G(\fy^1)(\ta_b^*\psi,q)\|_E
 \pr \le \La\left[\|\fy^0-\ta_b\fy^1\|_E + |b|\phi_\de(\|\fy^1\|_E)\right]^{\al,1}\left[\|\psi\|_E + |q|\phi_\de(\|\fy^1\|_E)\right],}
for all $\fy^0,\fy^1,\psi\in\HH$ and $q,b\in\R^d$, where 
\EQ{ \label{def double power}
 x^{\al,1}:=|x|^\al+|x|.}
\end{defn}

We will prove that $\G_{\ell,\de}^{\al,\La}$ is invariant by the flow, provided that $\de,\ell$ are small and $\La$ is large. First we investigate the backward evolution of the mobile derivative. Assuming the smallness of $\ell,\de>0$ as in \eqref{cond el-de} and \eqref{cond2 el-de}, for any $G\in\G_{\ell,\de}$ and $t>0$, define $G_t:\HH\to P_-\HH$ and $\widehat G_t:\HH\to\HH$ by 
\EQ{ \label{def hatG}
 G_t:=\U(t) G, \pq  \widehat G_t(\fy) := \fy_{\ge 0} + G_t(\fy).}
Let $\psi\in\HH$, $q,b\in\R^d$, and $t_0\in[0,1]$. For small $\e\in\R$, let  
\EQ{
 v_{(\e)}(t):=U(t-t_0)\widehat G_{t_0}(\fy_{(\e)}), 
 \pq w_{(\e)}(t,x)=v_{(\e)}(t,x-c_{(\e)}),} 
where $(w,c)=(w_{(\e)},c_{(\e)})$ is the solution of \eqref{eq wc} with the initial data 
\EQ{
 w_{(\e)}(t_0)=(\fy+\e\psi)(x-b), \pq c_{(\e)}(t_0)=b+\e q.}
Since the nonlinear term $(F,B)(w,c)$ in \eqref{eq wc} is $C^1$ from $\Str\times L^\I_t$ to $L^1_t\HH\times L^\I_t$, with a small factor on a short time interval $(0,T)$, it is straightforward by the iteration argument that $(w,c)$ is differentiable in $\Str\times L^\I_t$ at $\e=0$, with the derivative 
\EQ{ \label{def zg}
 \pt (z,g):=\lim_{\e\to 0}\frac{(w_{(\e)},c_{(\e)})-(w_{(0)},c_{(0)})}{\e},
 \pr \|z\|_{\Str(0,1)}+\|g\|_{L^\I(0,1)} \lec \|z(t_0)\|_\HH + |g(t_0)|.}
Let $\y:=z(t,x+c_{(0)}(t))$ and $\widehat F=F-(0,f'(Q_{c_{(0)}})w_1)$. Then 
\EQ{ \label{eq eta}
 \pt \dot \y = J\LL \y + (0,g\cdot\na f'(Q) v_1) + B(w,c)\cdot\na\y + \ta_{c}^* \p_{(w,c)}\widehat F(w,c)\cdot(z,g),
 \pr \y(t_0)=P_{\ge 0}(\psi+q\cdot\na\fy)-q\cdot\na\widehat G_{t_0}(\fy)+\Dm G_{t_0}(\fy)(\psi,q), \pq g(t_0)=q,}
where the subscript $(0)$ is omitted. 
Mobile-differentiating the identities
\EQ{
 \pt P_- v_{(\e)}(t_0) = G_{t_0}(\fy_{(\e)}), 
 \pq P_- v_{(\e)}(t) = G_t(v_{(\e)}(t_0))}
yields
\EQ{
 \pt P_-[ \y(t_0) + q \cdot \na \widehat G_{t_0}(\fy) ] = \Dm G_{t_0}(\fy)(\psi,q),
 \pr P_-[ \y(t) + g(t)\cdot \na v(t) ] = \Dm G_t(v(t))(\y(t),g(t)).} 
Since $G(\fy)=G(P_{\ge 0}\fy)$ and 
\EQ{
 P_{\ge 0}\fy_{(\e)} = P_{\ge 0}[(\fy_{\ge 0}+\e\psi_{\ge 0})(x+\e q)+\e q\cdot\na \fy_-] + o(\e),}
we have 
\EQ{ \label{mobdif proj}
 \Dm G(\fy)(\psi,q) = \Dm G(\fy_{\ge 0})(\psi_{\ge 0}+q\cdot P_{\ge 0}\na\fy_-,q).}

\begin{lem} 
Let $\ell,\de,\La>0$ satisfy \eqref{cond el-de}, \eqref{cond2 el-de} and 
\EQ{
 \La \gg \ell/\de.}
Then for any $G\in\G_{\ell,\de}^{\al,\La}$ and any $t_0\in[1/2,1]$, $\U(t)G\in \G_{\ell,\de}^{\al,\La}$.
\end{lem}
\begin{proof}
First of all, \eqref{mobdif Lip bd} is enough to have \eqref{mob C11} in the case where 
\EQ{
 \ell/\La \ll \|\fy^0-\ta_b\fy^1\|_E + |b|\phi_\de(\fy^1).}
Hence we may assume 
\EQ{
 \mb\fy^\pa \le \|\fy^0-\ta_b\fy^1\|_E + |b|\phi_\de(\fy^1) \lec \ell/\La \ll \de.}
Therefore we have either $\|\fy^0\|_\HH\simeq\|\fy^1\|\gg\de$ or $\|\fy^0\|_\HH+\|\fy^1\|_\HH\lec\de$. 

Next we investigate evolution of the difference of mobile-derivatives. For any $\fy^0,\fy^1,\psi\in \HH$ and $q,b\in\R^d$, let 
\EQ{
 \pt v^0_{(\e)}(t_0):=\widehat G_{t_0}(\ta^*_{\e q}(\fy^0+\e\psi)), \pq v^1_{(\e)}(t_0):=\widehat G_{t_0}(\ta^*_{\e q}(\fy^1+\e\ta_b^*\psi)),
 \pr v^j_{(\e)}(t):=U(t-t_0)v^j_{(\e)}(t_0), \pq 
 \pq w^j_{(\e)}(t):=\ta_{c_{(\e)}^j(t)}v^j_{(\e)}(t),}
for $j=0,1$, where $c_{(\e)}^j$ is the solution of 
\EQ{
 \dot c_{(\e)}^j = B(w_{(\e)}^j,c_{(\e)}^j), \pq c_{(\e)}^0(t_0)=\e q, \pq c_{(\e)}^1(t_0)=b+\e q.}
Let $(z^j,g^j)$ be the derivative at $\e=0$ of $(w_{(\e)}^j,c_{(\e)}^j)$: 
\EQ{
 z^j=\lim_{\e\to 0}\frac{w_{(\e)}^j-w_{(0)}^j}{\e}, \pq g^j=\lim_{\e\to 0}\frac{c_{(\e)}^j-c_{(0)}^j}{\e}.}
Henceforth the subscript $(0)$ will be omitted. The initial values are 
\EQ{ 
 \pt w^0(t_0)=\widehat G_{t_0}(\fy^0), \pq c^0(t_0)=0,
 \pq w^1(t_0)=\ta_b\widehat G_{t_0}(\fy^1), \pq c^1(t_0)=b,
 \pr z^0(t_0) = P_{\ge 0}\psi + q\cdot\check G_{t_0}(\fy^0)+\Dm G_{t_0}(\fy^0)(\psi,q), \pq g^0(t_0)=q,
 \pr z^1(t_0) = P_{\ge 0}^b\psi + \ta_b[q\cdot\check G_{t_0}(\fy^1)+\Dm G_{t_0}(\fy^1)(\ta_b^*\psi,q)], \pq g^1(t_0)=q,}
where $P^b_*:=\ta_bP_*\ta_b^*$, and $\check G_t:\HH\to\HH^d$ is defined by 
\EQ{
 \check G_t(\fy)=\na P_-\fy - P_-\na\fy - \na G_t(\fy).}
Let $\y^j(t,x)=z^j(t,x+c^j(t))$, then we have 
\EQ{
 \pt P_-[ \y^0(t_0) + q \cdot \na \widehat G_{t_0}(\fy^0) ] = \Dm G_{t_0}(\fy^0)(\psi,q),
 \pr P_-[ \y^1(t_0) + q \cdot \na \widehat G_{t_0}(\fy^1) ] = \Dm G_{t_0}(\fy^1)(\ta_b^*\psi,q),
 \pr P_-[ \y^j(t) + g^j(t) \cdot \na v^j(t) ] = \Dm G_t(v^j(t))(\y^j(t),g^j(t)).} 
Thus we obtain 
\EQ{ \label{C11 decompose}
 \pt \Dm G_{t_0}(\fy^0)(\psi,q) - \Dm G_{t_0}(\fy^1)(\ta_b^*\psi,q)
 = \diff[\y_-^\pa(t_0) + q\cdot P_-\na\widehat G_{t_0}(\fy^\pa)]
 \pr= e^{J\LL t_0}\diff \Dm G_0(v^\pa(0))(\y^\pa(0),g^\pa(0)) - e^{J\LL t_0} \diff  g^\pa(0)\cdot P_- \na v^\pa(0) 
 \prq + \diff[\y_-^\pa(t_0)-e^{J\LL t_0}\y_-^\pa(0)] + q\cdot \diff P_- \na\widehat G_{t_0}(\fy^\pa).}
The first term on the right of \eqref{C11 decompose} can be rewritten by using \eqref{mobdif proj}
\EQ{
 \pn\diff \Dm G_0(v^\pa)(\y^\pa,g^\pa)
 \pt=  
  \diff \Dm G_0(v_{\ge 0}^0)(\y_{\ge 0}^0 + \Rm v^0,g^0)
 \prq- \diff \Dm G_0(v_{\ge 0}^1)(\y_{\ge 0}^1 + \Rm v^1,g^0) + \Dm G_0(v^1)(\y^1,\diff g^\pa),}
where all functions are evaluated at $t=0$, and the operator $\Rm$ is defined by
\EQ{ \label{def R}
 \Rm:=g^0\cdot P_{\ge 0} \na P_-.}

We say that a component in \eqref{C11 decompose} is negligible if its norm in $E$ is much smaller than the right-hand side of \eqref{mob C11}. So is the last term in \eqref{C11 decompose}, since $\La\gg 1$ and 
\EQ{
 \|\diff P_- \na\widehat G_{t_0}(\fy^\pa)\|_E \lec \|D^{-1}\diff{\widehat G_{t_0}}(\fy^\pa)\|_\HH
 \lec \tm\fy^\pa \lec \|\fy^0-\ta_b\fy^1\|_E. }

In order to estimate the other terms, we prepare rough bounds on the unknowns. Lemma \ref{lem:bd v} together with $G\in\G_{\ell,\de}$ implies that 
\EQ{ \label{bd v^j}
 \|v^j\|_{\Str(0,1)} \lec \|\fy^j\|_E, \pq 
 \|v_{\ge 0}^j(0)\|_E \pt\le (1+C\ka+C\de)\|v_{\ge 0}(t_0)\|_E 
 \pr\le (1+C\ka+C\de+C\ell)\|\fy_{\ge 0}^j\|_E.}
The estimates in \eqref{est diff BF} together with $G\in\G_{\ell,\de}$ imply 
\EQ{ \label{bd diff wc}
 \pt\|\diff w^\pa\|_{\Str(0,1)}+\|\diff c^\pa\|_{L^\I(0,1)} 
 \pn\lec \|\diff w^\pa(t_0)\|_\HH + |\diff c^\pa(t_0)|
 \pr\lec \|\fy^0_{\ge 0}-\ta_b\fy^1_{\ge 0}\|_E + |b|\ell\|\fy^1\|_E + \ell( \tm\fy^\pa_{\ge 0} + |b| \phi_\de(\|\fy^1\|_E) )+ |b|
 \pr\lec \|\fy^0-\ta_b\fy^1\|_E + |b| \phi_\de(\|\fy^1\|_E).}
The equation for each $(z^j,g^j)$ is given as follows. 
\EQ{
 \pt \dot g=B', \pq \dot z = F', 
 \pr \chi_\de':=2\de^{-2}\chi'(\|w\|_\HH^2/\de^2)\LR{w,z}_\HH,
 \pq I:=(H(Q)-\LR{\na^2Q_c|w_1})^{-1},
 \pr B':=\chi_\de'I\frac{\om(w,\na w)}{2}+ I(\LR{\na Q_c|z_1}-g\LR{\na^2Q_c|w_1})B+\chi_\de I \om(z,\na w),
 \pr F':= \mat{B'\cdot\na Q_c + g\cdot \na^2 Q_c\cdot B \\ f'(Q_c)z_1 - g\cdot\na f'(Q_c)w_1 + N'},
 \pr N':=\chi_\de'N_c+\chi_\de[(f'(Q_c+w_1)-f'(Q_c))(z_1-g\cdot\na Q_c)+g\cdot\na f'(Q_c)w_1], }
where the superscript $j$ and the dependence on $(w,c)$ are omitted. Using the estimate $\|w\|_{\Str(-1,1)} \lec \|w(t_0)\|_\HH$, we obtain in the same way as for \eqref{est diff BF}, 
\EQ{
 \pt |\chi_\de'| \lec \|z\|_\HH/\de, \pq |B|\lec\de^2, 
 \pq |B'| \lec \de\|z\|_\HH + |g|\bde{\|w(0)\|_\HH}^3,
 \pr \|F'\|_{L^1\HH(-1,1)} \lec[T+\bde{\|w(t_0)\|_\HH}][\|z\|_\Str+\|g\|_{L^\I}]+T\|w(t_0)\|_\HH\|g\|_{L^\I}.}
In the case $\|\fy^0\|_\HH\simeq\|\fy^1\|_\HH\gg\de$, we have $g^j\equiv q$, and so by the Strichartz estimate, 
\EQ{
 \|z^j\|_{\Str} \lec \|z^j(t_0)\|_\HH + |q|\|w^j(t_0)\|_\HH.}
In the other case $\|\fy^\pa\|_\HH\lec\de$, 
\EQ{
 \|z^j\|_\Str+\|g^j\|_{L^\I} \lec \|z^j(t_0)\|_\HH + |g^j(t_0)|.}
In both cases, \eqref{mobdif Lip bd} implies 
\EQ{ \label{bd zg}
 \pt\|z^j\|_{\Str(0,1)}+\|g^j\|_{L^\I(0,1)} 
  \pr\lec \|\psi\|_\HH + |q|\ell\|\fy^j\|_\HH + \ell[\|\psi\|_\HH+|q|\phi_\de(\|\fy^j\|_E)]+|q| 
  \pr\lec \|\psi\|_E + |q| \phi_\de(\|\fy^j\|_E).}

For the difference estimate, we consider the two cases separately. If $\|\fy^0\|_\HH\simeq\|\fy^1\|_\HH\gg\de$, then $g^j\equiv q$ and 
\EQ{ \label{eq diff z lin}
 \diff{\dot z}=J\D\diff z + (0,\diff f'(Q_{c^\pa})z_1^\pa - q\cdot\na f'(Q)\diff w_1^\pa - q\cdot\diff{\na f'(Q_{c^\pa})}w_1^1),}
and so
\EQ{
 \|\diff z^\pa\|_{\Str(0,1)} \lec \|\diff z^\pa(t_0)\| + |b|\|z^\pa\|_\Str + |q|\|\diff w^\pa\|_\Str + |b|^\al|q|\|w^\pa\|_\Str,}
where the last term comes from the last one of \eqref{eq diff z lin}. Inserting \eqref{bd diff wc} and \eqref{bd zg}, we obtain 
\EQ{ \label{bd diff zg} 
 \pt\|\diff(z^\pa,g^\pa)\|_{\Str\times L^\I(0,1)}  
 \pr\lec [\|\fy^0-\ta_b\fy^1\|_E + |b|\phi_\de(\|\fy^1\|_E)]^{\al,1}[\|\psi\|+|q|\phi_\de(\|\fy^1\|_E)].}
In the nonlinear case $\|\fy^\pa\|_\HH\lec\de$, we have 
\EQ{
 \pt |\diff\chi_\de'| \lec \de^{-2}(\|\diff w^\pa\|_\HH \|z^\pa\|_\HH + \|w^\pa\|_\HH \|\diff z^\pa\|_\HH),
 \pr |\diff B'| \lec \|\diff(w^\pa,c^\pa)\|_{\HH\times\R^d}\|(z^\pa,g^\pa)\|_{\HH\times\R^d} + \|w^\pa\|_{\HH} \|\diff(z^\pa,g^\pa)\|_{\HH\times\R^d},
 \pr \|\diff F'\|_{L^1\HH} \lec T[\|\diff z^\pa\|_{\Str} + \de\|\diff c^\pa\|_{L^\I}^\al\|g\|_{L^\I}] 
 + \|w^\pa\|_\Str\|\diff(z^\pa,g^\pa)\|_{\Str\times L^\I}
 \prQQ\qquad+ \|\diff(w^\pa,c^\pa)\|_{\Str\times L^\I}\|(z^\pa,g^\pa)\|_{\Str\times L^\I},} 
where the term with $\al$ power comes from the same term as in the linear case, i.e., $(1-\chi_\de)q\cdot\diff\na f'(Q_{c^\pa})w_1^1$. Hence 
\EQ{
 \|\diff(z^\pa,g^\pa)\|_{\Str\times L^\I(0,1)}
 \lec \|\diff z^\pa(t_0)\|_\HH + [\|\diff w^\pa\|_\Str+ \|\diff c^\pa\|_{L^\I}^{\al,1}]\|(z^\pa,g^\pa)\|_{\Str\times L^\I},}
which, together with \eqref{bd diff wc} and \eqref{bd zg}, leads to the same bound \eqref{bd diff zg} as in the linear case.

For the penultimate term in \eqref{C11 decompose}, we obtain from the equation of $\y^j$ \eqref{eq eta} in the same way as above, 
\EQ{
 \pt\|\diff P_-[\y^\pa(t_0)-e^{J\LL t_0}\y^\pa(0)]\|_E
 \pr\lec \|\diff g^\pa\|_{L^\I_t}\|v^\pa\|_{\Str} 
   + \|g^\pa\|_{L^\I_t} \|\tm v^\pa\|_{L^\I_t}
 \prq+ \de\|\diff(z^\pa,g^\pa)\|_{\Str\times L^\I_t} + \|\diff(w^\pa,c^\pa)\|_{\Str\times L^\I_t} \|(z^\pa,g^\pa)\|_{\Str\times L^\I_t},}
where we do not get the term with $\al$-power, since the potential term is frozen in the $\y$ equation. Using \eqref{bd v^j}--\eqref{bd zg} and Lemma \ref{lem:diff v}, 
we can easily observe that the above is negligible because $\La\gg 1$. Here again we have used that either $\|\fy^\pa\|_\HH\lec\de$ or $g^j\equiv q$, 
which will be tacitly utilized in the following,  too. 
Hence the second term on the right of \eqref{C11 decompose} is also negligible by \eqref{bd v^j} and \eqref{bd diff zg}. 

For the remaining and leading term of \eqref{C11 decompose}, we have 
\EQ{ \label{C11 main est}
 \pt\|e^{J\LL t_0}\diff \Dm G_0(v^\pa)(\y^\pa,g^\pa)\|_E
 \pr\le 
   e^{-\uk t_0}\La[\|v_{\ge 0}^0-\ta_{\diff c^\pa} v_{\ge 0}^1\|_E+|\diff c^\pa|\phi_\de(\|v^1_{\ge 0}\|_E)]^{\al,1}
 \prQQ\times [\|\y_{\ge 0}^0 + \Rm v^0\|_E + |g^0|\phi_\de(\|v^1_{\ge 0}\|_E)]
 \prq + C\ell |\diff g^\pa|\phi_\de(\|v^1_{\ge 0}\|_E) + C\ell \|\ta_{\diff c^\pa}^*[\y_{\ge 0}^0 + \Rm v^0] - [\y_{\ge 0}^1 + \Rm v^1] \|_E,}
where $t=0$. The penultimate term in \eqref{C11 main est} is negligible thanks to $\ell\ll\La$, \eqref{bd v^j} and \eqref{bd diff zg}. The last term in \eqref{C11 main est} is dominated by 
\EQ{
 \pt \|\ta_{\diff c^\pa}^*\y^0-\y^1\|_\HH + \|[P_-,\ta_{\diff c^\pa}^*]\y^0\|_\HH + |\diff c^\pa|\|\Rm v^0\|_\HH + \|\diff \Rm v^\pa\|_\HH
 \pr\lec \|\diff z^\pa\|_\HH + |\diff c^\pa|\|z^0\|_\HH + |\diff c^\pa||g^0|\|v^0\|_\HH + |g^0|\tm v^\pa.}
Hence it is also negligible by using the estimates \eqref{bd v^j}--\eqref{bd diff zg} and $\ell\ll\La$. 

It remains to deal with the leading term of \eqref{C11 main est}, for which we need more precise estimates, employing the time decay of $e^{-\uk t_0}$. 
First we consider the linearized case $\|\fy^0\|_\HH\simeq\|\fy^1\|_\HH\gg\de$, using the equations 
\EQ{
 \pt v^j(0)=e^{-J\LL t_0}\widehat G_{t_0}(\fy^j), \pq c^0(t)=0, \pq c^1(t)=b,
 \pr \dot\y^j=J\LL \y^j + (0,q\cdot \na f'(Q) v^j_1), \pq g^j(t)=q.}
The first component on the right of \eqref{C11 main est} is estimated by 
\EQ{
 \pt\|v^0_{\ge 0}-\ta_{\diff c^\pa} v^1_{\ge 0}\|_E
 \pr\le \|e^{-J\LL t_0}P_{\ge 0}(\fy^0-\ta_b\fy^1)\|_E + \|e^{-J\LL t_0}[P_-,\ta_b]\fy^1\|_E + \|[e^{-J\LL t_0},\ta_b]\fy^1_{\ge 0}\|_E
 \pr\le e^{\ka}\|\fy^0-\ta_b\fy^1\|_E + C|b|\|\fy^1\|_E,}
and the third component by 
\EQ{
 \|\y_{\ge 0}^0 + \Rm  v^0\|_E 
 \le \|e^{-J\LL t_0}\y_{\ge 0}^0(t_0)\|_E + C|q|\|\fy^0\|_E 
 \le e^{\ka}\|\psi\|_E + C|q|\|\fy^0\|_E.}
Since $\uk t_0 \gg \ka$ and $\phi_\de(\|\fy^j\|_E)\gg\|\fy^j\|_E$, we see that the leading term of \eqref{C11 main est} is smaller than 
\EQ{
 \frac 12 \La \left[\|\fy^0-\ta_b\fy^1\|_E + |b|\phi_\de(\|\fy^1\|_E)\right]\left[\|\psi\|_E + |q|\phi_\de(\|\fy^1\|_E)\right].}
Hence the other terms, which have been shown to be negligible, are absorbed into the remaining half of \eqref{mob C11}. Thus the linearized case $\|\fy^\pa\|\gg\de$ is done. 

It remains to consider the nonlinear case $\|\fy^\pa\|\lec\de$. Let 
\EQ{
 \x:=v^0-\ta_{\diff c^\pa} v^1=\ta_{c^0}^*\diff w^\pa,} 
then we have 
\EQ{
 \pt v^0_{\ge 0}-\ta_{\diff c^\pa} v^1_{\ge 0} 
   = \x_{\ge 0} + [P_-,\ta_{\diff c^\pa}]v^1, 
 \pr \dot\x = J\LL\x + B(w^0,c^0)\cdot\na\x + \ta_{c^0}^*[(0,\diff f'(Q_{c^\pa})w^1_1) + \diff{\widehat F}(w^\pa,c^\pa)].}
The commutator term is negligible thanks to \eqref{bd v^j}, \eqref{bd diff wc} and $\|\fy^1\|_\HH\lec\de$. For the discrete spectral part of $\x_{\ge 0}$, we have 
\EQ{
 \pt \|P_{0+}[\x-e^{-J\LL t_0}\x(t_0)]\|_E
    \lec  \de \|\diff (w^\pa,c^\pa)\|_{\Str\times L^\I_t(0,1)},
 \pr \|e^{-J\LL t_0}\x_{0+}(t_0)\|_E \le e^{\ka}\|P_{0+}\x(t_0)\|_E,}
and for the continuous spectral part, 
\EQ{
 \p_t\LR{\LL\x|P_\ga\x}/2
  \pt= B(w^0,c^0)[\LR{f'(Q)\x_1|\na\x_1}-\LR{\LL\x|P_d\na\x}]
  \prq+ \LR{\LL\x|P_\ga\ta_{c^0}^*[(0,\diff f'(Q_{c^\pa})w^1_1) + \diff{\widehat F}(w^\pa,c^\pa)]}
 \pr\lec \de^2\|\x\|_E^2 + \|\x\|_E \de \|\diff(w^\pa,c^\pa)\|_{\Str\times L^\I_t(0,1)} .}
Thus we deduce 
\EQ{ \label{bd xi}
 \|\x_{\ge 0}(0)\|_E \pt\le e^{\ka}\|\x_{\ge 0}(t_0)\|_E + C\sqrt\de\|\diff(w^\pa,c^\pa)\|_{\Str\times L^\I_t(0,1)} 
 \pr\le (1+C\ka+C\sqrt\de) [\|\fy^0-\ta_b\fy^1\|_E + |b|\phi_\de(\|\fy^1\|_E)].}
Similarly for the $\y^0$ component, \eqref{eq eta} implies 
\EQ{
 \pt \|\y^0(t_0)\|_E \le \|\psi\|_E + C|q|\|\fy^0\|_E + C\ell(\|\psi\|_E+|q|), 
 \pr \|P_{0+}[\y^0(0)-e^{-J\LL t_0}\y^0(t_0)]\|_E
  \lec \de \|(z^0,g^0)\|_{\Str\times L^\I_t},
 \pr \p_t\LR{\LL\y^0|P_\ga\y^0}/2
  = B(w^0,c^0)[\LR{\na f'(Q)\y^0_1|\na\y^0}-\LR{\LL\y^0|P_d\na\y^0}]
 \prQQ + \LR{\LL P_\ga\y^0|(0,g^0\cdot\na f'(Q)v^0_1)+\ta_{c^0}^*[\p_{(w,c)}\widehat F(w^0,c^0)\cdot(z^0,g^0)]},}
and so, using \eqref{bd zg} we obtain
\EQ{ \label{bd eta_>0}
 \|\y^0_{\ge 0}(0)\|_E \le (1+C\ka+C\sqrt\de+C\ell)\|\psi\|_E  + C|q|(\de+\ell).}
Also using \eqref{bd v^j}--\eqref{bd zg}, we have 
\EQ{ 
 \pt \|\Rm v^0\|_E \lec |g^0|\|v^0\|_E \lec \de(\|\psi\|_E+|q|),
 \pr \phi_\de(\|v^1_{\ge 0}\|_E) \le (1+C\ka+C\de+C\ell)\phi_\de(\|\fy^1\|_E),
 \pr |\diff c^\pa(0)| \le |b| + C\|\diff B(w^\pa,c^\pa)\|_{L^\I_t} \le |b| + C\de[\|\fy^0-\ta_b\fy^1\|_E+|b|],
 \pr |g^0| \le |q| + C\|B_w(w^0,c^0)z^0+B_c(w^0,c^0)g^0\|_{L^\I_t}
 \le |q| + C\de[\|\psi\|_E+|q|]}

Putting \eqref{bd xi}, \eqref{bd eta_>0} and the above estimates together, and using $\ka+\sqrt\de+\ell \ll \uk$, we see that the leading term in \eqref{C11 main est} is bounded by  
\EQ{
 \frac{1}{2}\La \left[\|\fy^0-\ta_b\fy^1\|_E + |b|\phi_\de(\|\fy^1\|_E)\right]^{\al,1}\left[\|\psi\|_E + |q|\phi_\de(\|\fy^1\|_E)\right],}
so the remaining half can absorb the negligible terms, concluding the proof in the nonlinear case $\|\fy^\pa\|_E\lec\de$. 
\end{proof}

Mobile-differentiability of the fixed point $G_*\in\G_{\ell,\de}$ now follows from the closedness of $\G_{\ell,\de}^\La$ for pointwise convergence. 
\begin{lem}
Under the assumption of the above lemma, let $G_n\in\G_{\ell,\de}^{\al,\La}$ be a sequence of maps such that $G_n(\fy)\to G(\fy)$ as $n\to\I$ for all $\fy\in\HH$. Then $G\in\G_{\ell,\de}^{\al,\La}$. 
\end{lem}
\begin{proof}
Since $G_n\in\G_{\ell,\de}$, \eqref{mobdif Lip bd} implies that $\Dm G_n(\fy)$ for each $\fy\in\HH$ is bounded in $(\HH\times\R^d)^*$. Hence after extracting a subsequence, we have weak convergence of $\Dm G_n(\fy)$ in $(\HH\times\R^d)^*$ for $\fy$ in a dense countable subset $A\subset \HH$. To extend the convergence to all $\fy\in\HH$, take a sequence $\fy_n$ converging to $\fy$ in $\HH$. Then for any $\psi\in\HH$ and $q\in\R^d$, we have
\EQ{
 \pt\|\Dm G_k(\fy_n)(\psi,q)-\Dm G_k(\fy_m)(\psi,q)\|_E
 \pr\le \La \|\fy_n-\fy_m\|^{\al,1}[\|\psi\|_E+|q|\phi_\de(\|\fy_n\|_E)]
 \to 0,}
as $n,m\to\I$, uniformly for all $k\in\N$. Hence we have the convergence
\EQ{
 \lim_{k\to\I}\Dm G_k(\fy)=\lim_{k\to\I}\lim_{n\to\I}\Dm G_k(\fy_n)
 = \lim_{n\to\I}\lim_{k\to\I}\Dm G_k(\fy_n)}
weakly in $(\HH\times\R^d)^*$. To see the mobile differentiability of $G$, we use the mean value theorem. For any $\fy,\psi\in\HH$, $q\in\R^d$, $k\in\N$ and $\e\in\R$ small, there is $\te\in[0,1]$ such that 
\EQ{
 G_k(\fy_{(\e)})-G_k(\fy)=\e \Dm G_k(\fy^{\te\e})(\psi,q).}
Since $G_k\in\G_{\ell,\de}^{\al,\La}$, we have 
\EQ{
 \pt\|\Dm G_k(\fy^{\te\e})(\psi,q)-\Dm G_k(\fy)(\psi,q)\|_\HH
 \pr\lec \La\|\fy^{\te\e}-\fy\|_\HH^{\al,1}[\|\psi\|_\HH+|q|\phi_\de(\|\fy\|_E)]
 \pn\to 0}
as $\e\to 0$, uniformly for all $k\in\N$. Hence the limit $G(\fy)$ is mobile differentiable and 
\EQ{
 \Dm G(\fy)(\psi,q) = \lim_{k\to\I} \Dm G_k(\fy)(\psi,q),}
which implies $G\in\G_{\ell,\de}^{\al,\La}$. 
\end{proof}

Therefore the fixed point $G_*$ belongs to $\G_{\ell,\de}^{\al,\La}$ and so in particular $C^{1,\al}$ in the $\HH$ topology. Then it is easy to see that $\ti G_*$ is also $C^{1,\al}$, so are $\M_{cu,0}$ and $\M_{cu}$.

\appendix 
\section{Table of Notation}
{\small
\begin{longtable}{l|l|l}
 \hline 
 $\bigcirc^\pa$, $\diff{\bigcirc^\pa}$ & ordered pair and difference & \eqref{def diff} \\
 $\bde{\bigcirc}$ & minimum & \eqref{def min} \\ 
 $\bigcirc^{\al,1}$ & sum of two powers & \eqref{def double power} \\
 $B_R(\bigcirc)$ & a ball in the Banach space & \eqref{def BR} \\
 $D,J,\D$ & basic operators & \eqref{def L+}, \eqref{def J}, \eqref{def D} \\ 
\hline 
 $f$, $d$, $p$ & nonlinearity, dimension and power & \eqref{asm f}, \eqref{asm f'} \\
 $N(v),\vec N(v),N_c(w)$, $C(v)$ & higher order nonlinearity around $Q$ & \eqref{eq0 v},\eqref{def N}, \eqref{def ANc}, \eqref{def C} \\ 
 $A(v),A_c(w)$ & transport terms & \eqref{eq1 v}, \eqref{def ANc} \\
 $\chi_\de(v)$ & localizer around $0$ in $\HH$ & \eqref{def chide} \\ 
 $M(v),M_\de(v)$ & nonlinearity for $v$ & \eqref{def M}, \eqref{eq v} \\
 $B(w,c),F(w,c)$ & nonlinearity for (c,w)& \eqref{def BF} \\
 $U(t)$, $\U(t)$ & localized flow in $\HH$ and on graphs & \eqref{def U}, Lemma \ref{lem:U on graph} \\ 
\hline
 $E(u),P(u)$ & energy and momentum & \eqref{def EP}, \eqref{def P} \\
 $\HH,\HH_0,\HH_\perp$  & energy space and its subsets & \eqref{def HH}, \eqref{def H0}, \eqref{def Hperp} \\
 $\|\cdot\|_E$, $\ka$ & energy norm and the parameter & \eqref{def E}, \eqref{def ka}\\ 
 $\LR{\cdot|\cdot}$, $\LR{\cdot,\cdot}_\HH$, $\om(\cdot,\cdot)$ & bilinear forms on $\HH$ & \eqref{def prdH}, \eqref{def prdL2}, \eqref{def om} \\ 
 $\|\cdot\|_\Str$ & Strichartz norm & \eqref{def Str} \\
\hline
 $Q,Q_c,Q(p,q),$ & the static solution, its transforms,  & \eqref{static eq}, \eqref{def Qc}, \eqref{def Qpq}\\ 
 $\vec Q,\vec Q(p,q),\Sm(Q),\Sm_0(Q)$ & vector forms, and the families & \eqref{def vQ}, \eqref{def sol mfd}, \eqref{def S0Q} \\ 
 $H(Q)=H_{\al\be}(Q)$ & kinetic energy matrix of $Q$ & \eqref{def H} \\
\hline
 $L_+,\LL$ & linearized operators at $Q$ & \eqref{def L+},\eqref{def LL} \\
 $k,K,\uk,\ok$ & eigenvalues and their bounds & \eqref{spec L+}, \\
 $\rh_k,g_{k\pm}$ & eigenfunctions of $L_+,\LL$ & \eqref{def rk},\eqref{def gk} \\
 $\la_{k\pm},\mu,\nu,\ga,v_\pm,v_d\etc$ & spectral components of $v$ & \eqref{spec decp}, \eqref{spec combi} \\
 $P_{\pm},P_0,P_d,P_{\mu},P_\ga\etc$ & the corresponding operators & \eqref{spec decp}, \eqref{spec combi} \\
\hline
 $\mb$, $\tm$, $\phi_\de$ & mobile distances and the cut-off & \eqref{def mobile}, \eqref{def tm}, \eqref{def phi} \\ 
 $\ta_c,\ta^j$ & translation operators & \eqref{def ta}, \eqref{def zej} \\
\hline
 $\G_{\ell,\de},\G$, $\G_{\ell,\de}^{\al,\La}$ & sets of Lipschitz maps & \eqref{def Gelde}, \eqref{def Gnorm}, Definition \ref{def mob C1al} \\
 $\gr{G}$ & the graph in $\HH$ & \eqref{def graph} \\ 
 $G_*$ & the invariant map & Theorem \ref{mfd loc eq} \\
 $\ti G$ & projection to $\HH_\perp$ & Lemma \ref{lem:tiG} \\
\hline
 $\Dm G$ & mobile derivative & Definition \ref{def mob diff} \\
 $\fy_{(\e)}$ & translating variation in $\HH$ & Definition \ref{def mob diff} \\
 $\widehat G_t$ & complemented graphs & \eqref{def hatG} \\ 
 $\Rm$ & spectral error of the $g$-translation & \eqref{def R} \\
 
 \hline 
\end{longtable}}

\end{document}